\documentclass[11pt]{amsart}

\usepackage{amsmath,amsthm}
\usepackage{amsfonts}
\usepackage{verbatim}
\usepackage{amssymb}
\usepackage{amscd}
\usepackage{bbm}
\usepackage{mathrsfs}
\usepackage[nobysame,non-sorted-cites]{amsrefs}
\usepackage[citecolor=red,colorlinks=true]{hyperref}
\usepackage{fouriernc}

\usepackage{tikz}
\usetikzlibrary{matrix,arrows,decorations.pathmorphing}

\usepackage{color}

\long\def\comment#1\endcomment{}

\def\R{{\mathbb R}}

\def\Z{{\mathbb Z}}
\def\N{{\mathbb N}}

\def\calt{{\mathcal{T}}}

\def\calp{{\mathcal{P}}}
\def\q{{\mathfrak{q}}}

\newcommand{\set}[1]{\left\{#1\right\}}
\newcommand{\setd}[2]{\left\{#1\ \colon\ #2\right\}}

\newtheorem{theorem}{Theorem}[section]

\newtheorem{corollary}[theorem]{Corollary} 
\newtheorem{lemma}[theorem]{Lemma} 
\newtheorem{proposition}[theorem]{Proposition} 
\newtheorem{definition}[theorem]{Definition} 
\newtheorem{question}[theorem]{Question}

\newtheorem{remark}[theorem]{Remark}
\newtheorem{remarks}[theorem]{Remarks} 

\newtheorem{examples}[theorem]{Examples}  
\newtheorem{conjecture}[theorem]{Conjecture}

\newtheorem{problem}[theorem]{Problem}

\newcommand{\supp}{\operatorname{supp}}

\newcommand{\RR}{\mathbb{R}}
\newcommand{\ZZ}{\mathbb{Z}}
\newcommand{\CC}{\mathbb{C}}
\newcommand{\NN}{\mathbb{N}}

\newcommand{\opi}{\overline{\pi}}

\newcommand{\iip}{(i)}
\newcommand{\dist}{\operatorname{dist}}

\renewcommand{\theenumi}{\roman{enumi}} 

\title{Kazhdan projections, random walks and ergodic theorems}

\date{\today}

\author{Cornelia Dru\c{t}u}

\thanks{The research of both authors was supported by the EPSRC grant ``Geometric and analytic aspects of infinite groups". The research of the first author was also partially supported by the project ANR Blanc ANR-10-BLAN 0116, acronym GGAA, and by the Labex CEMPI (ANR-11-LABX-0007-01). The research of the second author was partially supported by Narodowe Centrum Nauki grant DEC-2013/10/EST1/00352.}

\address{{\normalfont{C.Dru\c{t}u:}} Mathematical Institute, University of Oxford, United Kingdom.}
\email{Cornelia.Drutu@maths.ox.ac.uk}

\author{Piotr W. Nowak}
\address{{\normalfont{P. W. Nowak:}} Institute of Mathematics of the Polish Academy of Sciences, Poland -- \and -- Institute of Mathematics, University of Warsaw, Poland}
\email{pnowak@impan.gov.pl}

\begin{document}

\newcounter{aux1}
\newcounter{aux2}
\newcounter{mainthm}
\newcounter{expthm}
\newcounter{ergthm}
\newcounter{warpthm}

\maketitle

\begin{abstract}
In this paper we investigate generalizations of Kazhdan's property $(T)$ to the setting of uniformly convex Banach spaces. We explain the interplay between the existence of spectral gaps and that of Kazhdan projections.
Our methods employ Markov operators associated to a random walk on the group, for which we provide new norm estimates and convergence results. 
They exhibit useful properties and flexibility, and allow
to view Kazhdan projections in Banach spaces as natural objects associated to random walks on groups. 

We give a number of applications of these results. In particular, we address several open questions. 
We give a direct comparison of properties $(TE)$ and $FE$ with Lafforgue's reinforced Banach property $(T)$; 
we obtain shrinking target theorems for orbits of Kazhdan groups; finally, answering a question of Willett and Yu we construct 
non-compact ghost projections for warped cones. In this last case we conjecture that such warped cones provide 
 counterexamples to the coarse Baum-Connes conjecture.
\end{abstract}

\section{Introduction}

One way to investigate properties of groups, especially with a view to their actions on Banach spaces, is through the group Banach algebras. These are natural analytic objects encoding many properties of the group. The existence of projections in such algebras is a particularly important and challenging problem. For instance, the (non)existence of idempotents other than 0 and 1 in the reduced group  $C^*$-algebra of a torsion-free group is a long-standing conjecture of Kadison and Kaplansky. When the group is amenable (and more generally, a-T-menable) the Kadison-Kaplansky conjecture is known to be true. Additionally, for amenable and torsion free groups the maximal group $C^*$-algebra is isomorphic to the reduced group $C^*$-algebra, therefore the maximal group $C^*$-algebra does not have non-trivial idempotents either. 

A main result of this paper is an explicit construction of proper idempotents in many group Banach algebras. The construction is based on random walks. The relation 
between Kazhdan projections and random walks can be extracted from the spectral considerations regarding Kazhdan projections for discrete groups in the Hilbert space setting, see \cite{valette,harpe-robertson-valette}. However, spectral theory is not available in our setting.
Our construction turns out to be relevant in various contexts, from expander graphs to ergodic geometry and the Baum-Connes conjecture.

A \emph{Kazhdan projection} for a locally compact group $G$ is an idempotent in the maximal group $C^*$-algebra $C^*_{\max}(G)$, whose image under 
any unitary representation
is the projection onto the space of invariant vectors. Such projections exist in $C^*_{\max}(G)$ 
if and only if the group $G$ has Kazhdan's property $(T)$ \cite{akemann-walter}. They are important for many applications. A classical consequence of their existence is the fact that the
map on $K$-theory induced by the canonical
homomorphism $C^*_{\max}(G)\to C^*_{r}(G)$ from the maximal to the reduced group $C^*$-algebra, fails to be an isomorphism for Kazhdan groups,
see  e.g. \cite[Ch. 2, S. 4]{connes}.
They
play the main role in the
failure of some versions of the Baum-Connes conjecture, since projections of this type can often be shown not to be in the image of the Baum-Connes assembly map \cite{higson-lafforgue-skandalis}.
Kazhdan projections are the main ingredient of Lafforgue's reinforced Banach property $(T)$ and
allowed for the construction of the first examples of expanders with no coarse
embedding into any uniformly convex Banach space \cite{lafforgue-duke}.
Finally, they also play an important role in the generalization of
property $(T)$ to $C^*$-algebras \cite{brown}.

\subsection*{Spectral gaps, Markov operators and projections}


At the core of our paper is a study of a Banach space version of property (T), formulated in a very general setting: with respect to a given family of isometric representations on Banach spaces. We prove that such a property can be characterized in three different ways: the standard spectral gap property, the behavior of the Markov operator on a canonical complement of the fixed vectors subspace, and the existence of a Kazhdan projection, with an explicit formula to calculate it, using Markov operators. 

Indeed, given an isometric representation 
$\pi$ of a group $G$ on a reflexive Banach space $E$, the subspace $E^\pi$ of fixed vectors has a canonical
$\pi$-complement, $E_\pi$ (see Section \ref{subsection: groups and representations} for details). 
Given a probability measure $\mu$ on $G$, let $A_\pi^\mu$ denote the Markov (averaging) operator associated to $\pi$ via $\mu$. We prove the following.

\setcounter{mainthm}{\value{theorem}}
\begin{theorem}\label{theorem in intro: spectral gaps and projections}
Let $G$ be a locally compact group, and $\mathcal{F}$ a family of isometric representations of $G$ on a  uniformly convex family $\mathcal{E}$ of Banach spaces. 
The following conditions are equivalent:
\begin{enumerate} \item the family $\mathcal{F}$ has a spectral gap (see Definition \ref{definition: spectral gaps});\label{theorem in intro: spectral gap}
\item there exists a compactly supported probability measure $\mu$ on $G$ and $\lambda<1$ \label{theorem in intro: norm of Markov}
such that for every isometric representation $\pi \in \mathcal{F}$ of $G$ on $E\in\mathcal{E}$ we have $\left\lVert A_\pi^\mu\vert_{E_\pi}\right\rVert<\lambda$;
\item there exists a compactly supported probability measure $\mu$ on $G$ and a number $\mathfrak{S}<\infty$ such that for every $\pi\in \mathcal{F}$ 
the iterated Markov operators $\left(A_\pi^\mu\right)^k$ 
converge with speed summable to at most $\mathfrak{S}$ to the projection $\mathcal{P}_\pi$ onto $E^\pi$ along $E_\pi$, that is
$$
\left\| \left(A_\pi^\mu\right)^k - {\mathcal{P}}_\pi \right\|\leq a_k\, , 
$$ 
where $\sum_{k} a_k \leq {\mathfrak{S}}$. \label{theorem in intro: condition convergence}
\end{enumerate}
\end{theorem}

In Theorem \ref{theorem: definition of the projection onto invariant vectors} we give an explicit formula for the projection $\mathcal{P}_\pi$ in terms of the Neumann series of the Markov operator:
\begin{equation}\label{formula:Neumannser}
\mathcal{P}_\pi = I_E- \left(\sum_{n=0}^{\infty} \left(A_{\pi }^{\mu }\right)^n \right) \left(I_E-A_\pi^\mu\right).
\end{equation}

The hypothesis of uniform convexity is needed only in the implication \eqref{theorem in intro: spectral gap} $\Rightarrow$\eqref{theorem in intro: norm of Markov} and \eqref{theorem in intro: condition convergence}, for the other implications it suffices to have a family of complemented representations on Banach spaces, in the sense of Definition \ref{def:compl}.

When $G$ has Kazhdan's property $(T)$ and $E$ is a Hilbert space, Theorem \ref{theorem in intro: spectral gaps and projections} holds for $\mathcal{F}$ the family of all unitary representations of $G$. However, as Theorem \ref{theorem in intro: spectral gaps and projections} is formulated in terms of a family of representations, it also applies in the setting of Property $(\tau )$ (see Section \ref{section: tau and expanders for Banach} and the corresponding paragraph later in the Introduction), of property $(T\ell_p)$ introduced in \cite{bekka-olivier} etc.

The equivalence in Theorem \ref{theorem in intro: spectral gaps and projections} has an effective side to it, described below.
Given a Kazhdan pair $(Q, \kappa )$ defining the spectral gap  
(see Definition \ref{definition: spectral gaps}), 
the conditions \eqref{theorem in intro: norm of Markov} and \eqref{theorem in intro: condition convergence} 
hold for a large class of measures, which we call \emph{admissible with respect to the Kazhdan set $Q$}. These are explicitly constructed by means of $Q$, see Definition \ref{def:admiss}; their particular construction is motivated by the sought after connection with Markov operators, which cannot function for general measures, see Remark \ref{rmk:referee}. For every such measure $\mu$, a constant $\lambda$ as in \eqref{theorem in intro: norm of Markov} can be computed using the Kazhdan constant $\kappa$, 
the modulus of uniform convexity of the family $\mathcal E$ and the choice of an appropriate compactly supported function on $G$ associated to $\mu$. Property \eqref{theorem in intro: condition convergence} then holds with $a_k =\lambda^k$. Conversely, given a measure $\mu $ and $\lambda \in (0,1)$ satisfying either \eqref{theorem in intro: norm of Markov} or \eqref{theorem in intro: condition convergence} with $a_k=\lambda^k$, the support of $\mu$ is a Kazhdan set and its corresponding Kazhdan constant is $1-\lambda $. This implication applies, for instance, in the case of semisimple Lie groups with finite center, and their unitary representations on Banach spaces, to any probability measure with symmetric support not contained in a closed amenable subgroup \cite[Theorem C]{shalom-kazhdan}.

One of the advantages of Theorem \ref{theorem in intro: spectral gaps and projections} 
is the high degree of flexibility in ensuring uniformity of several parameters for classes of
isometric representations. For instance, in the case of groups admitting finite Kazhdan sets (see Section \ref{sec:Kazhdansets}) this uniformity depends only on three quantities:
\begin{enumerate}
\renewcommand{\theenumi}{\alph{enumi}}
\item the Kazhdan constant of the family of representations,
\item the cardinality of the Kazhdan set $Q$,
\item the modulus of uniform convexity of the Banach spaces.
\end{enumerate}
It does not even depend on the group $G$, as long as we can arrange the above three items to have uniform bounds.


For applications, the existence of finite Kazhdan sets is a considerable asset: the averages become finite, the random walks discrete and an algorithmical approach and the use of computer become possible (see for instance Theorem \ref{theorem: definition of the projection onto invariant vectors}, Remark \ref{rem:algo} and Section \ref{sec:ergodic}). As it turns out, the existence of such finite sets is ensured in many cases outside the class of finitely generated groups, and in many cases the sets are described explicitly, as explained briefly in Section \ref{sec:Kazhdansets}.

\subsection*{Kazhdan projections in group Banach algebras and Lafforgue's reinforced Banach property $(T)$}

The uniform convergence described in Theorem \ref{theorem in intro: spectral gaps and projections}, \eqref{theorem in intro: condition convergence}, depending on the Kazhdan constant, the modulus of uniform convexity of the family $\mathcal E$, and the choice of the measure $\mu$, shows that the existence of a Kazhdan projection in  group Banach algebras is a consequence
of a uniform version of property $(TE)$. Property $(TE)$ was introduced and studied in \cite{fisher-margulis,bfgm} as a natural generalization of property $(T)$ from Hilbert to Banach spaces. 

\begin{theorem}[see Theorem \ref{theorem: Kazhdan projections in algebras defined via a class of reps} and Corollary \ref{corollary: Kazhdan projection and uniform (TE)}]\label{theorem: introduction Kazhdan projections in Lp algebra}
Let $G$ be a locally compact group and let $\mathcal{F}$ be a family of isometric representations of $G$ on a uniformly convex family of  Banach spaces.  
There exists a Kazhdan projection $p\in C_{\mathcal{F}}(G)$ if and only if the family $\mathcal{F}$ has a spectral gap.

In particular, if $G$ has Kazhdan's property $(T)$ then there exists a Kazhdan projection in the $L_p$-maximal group algebra $C_{\max}^p(G)$ for
every $1<p<\infty$.
\end{theorem}

Here, $C_{\mathcal{F}}(G)$ is a natural version of the maximal $C^*$-algebra of $G$ for the family $\mathcal{F}$ of representations, see Definition \ref{def:cf}. 
Banach group algebras for larger than isometric classes of representations were
introduced and studied by V.~Lafforgue \cite{lafforgue-duke}.

In the case of Hilbert spaces and unitary representations two previous proofs are known of the general relation between property $(T)$ and the existence of Kazhdan projections.
The first is due to Akemann and Walter \cite{akemann-walter}
and it relies on positive definite functions, a tool available essentially only in Hilbert spaces. The second proof, using minimal projections in $C^*$-algebras, is due to Valette \cite{valette}. Our approach shares some elements with the latter proof. 
Namely, Valette shows that property $(T)$ is equivalent to 1 being isolated in the spectrum of the Markov operator $\pi(u)$, for every unitary representation $\pi$, 
where $u$ is the uniform measure on a symmetric
generating set. The spectral projection associated to 1 is then the Kazhdan projection, and it follows that $u$, as an element of the maximal group $C^*$-algebra,
can be realized as a norm limit of explicit polynomials in $u$. In our proof, the existence of the Kazhdan projection in the algebra is the consequence of the 
appropriate convergence of a random walk.
Another construction of Kazhdan projections, yielding an explicit sequence of compactly supported functions approximating the projection, 
but limited to a specific class of Lie groups, has been provided by V. Lafforgue \cite{lafforgue-duke}.

The topic of operator algebras on $L_p$-spaces is an emerging direction in non-commutative geometry.
For certain groups the construction of Lafforgue \cite{lafforgue-duke} provided Kazhdan projections for maximal $L_p$-group algebras.  
Additionally, an approach to the Novikov conjecture via $L_p$-versions of the Baum-Connes conjecture has been recently developed by Kasparov and Yu.
Theorem \ref{theorem: introduction Kazhdan projections in Lp algebra} shows that for groups with property $(T)$ the same obstructions 
as in the Hilbert space case are likely to exist in $K$-theory.

Theorem \ref{theorem in intro: spectral gaps and projections} allows to compare V.~Lafforgue's definition of reinforced Banach property $(T)$ \cite{lafforgue-duke}
to other generalizations of property $(T)$ to Banach spaces, i.e. properties $(TE)$ and $FE$  \cite{fisher-margulis,bfgm}. The question of such a comparison has been considered by several experts previously. In \cite{lafforgue-jta} it was shown that the reinforced Banach property $(T)$ implies $FE$, and
in  \cite{bfgm} it was shown that property $FE$ implies $(TE)$. Since Lafforgue's reinforced Banach property $(T)$ is formulated in terms of existence of Kazhdan projections
in certain group Banach algebras, Theorem \ref{theorem: introduction Kazhdan projections in Lp algebra} provides implications in the other
direction. We discuss this in detail in Section \ref{section: relations between different notions of T}.

\medskip

As Theorem \ref{theorem in intro: spectral gaps and projections} holds for a family of representations, it can be applied in the context of property $(\tau)$, introduced by Lubotzky. Thus, we use Theorem \ref{theorem in intro: spectral gaps and projections} to formulate a generalization of property $(\tau)$ to uniformly convex Banach spaces, which is consistent with a notion of expanders for Banach spaces defined using Poincar\'e inequalities (Definition \ref{def:e-expanders}).   
 
 More precisely, for a uniformly convex Banach space $E$
we introduce property $(\tau_E)$ by the same definition as for Hilbert spaces, requiring that
certain isometric representations factoring through finite quotients of $G$ are separated from the trivial representation; that is, they have
a uniform spectral gap. The following is a consequence of Theorem \ref{theorem in intro: spectral gaps and projections}.

\setcounter{expthm}{\value{theorem}}
\begin{theorem}\label{theorem: expanders theorem in the text}\label{thm:introExp}
Let $E$ be a uniformly convex Banach space, let $G$ be a finitely generated residually finite group 
and let ${\mathcal{N}}=\set{N_i}$ be a collection of finite index subgroups with trivial intersection.
The following conditions are equivalent:
\begin{enumerate}
\item $G$ has property $(\tau_E)$ with respect to ${\mathcal{N}}=\set{N_i}$ and a symmetric Kazhdan set $Q$; \label{enumi: tau_E}
\item the Cayley graphs $\operatorname{Cay}(G/N_i,Q)$ form a sequence of $E$-expanders;\label{enumi: E-expanders}
\item there exists a Kazhdan projection $p\in C_{\mathcal{N}(E)}(G)$.\label{enumi: Kazhdan projection}
\end{enumerate}
\end{theorem}

In the Hilbert space case the algebra appearing in
the condition \eqref{enumi: Kazhdan projection} is a $C^*$-algebra. In that case the existence of the Kazhdan projection appearing in the above theorem was hinted at in \cite{higson,higson-lafforgue-skandalis}. It is \textit{via} such projections that the main existing counter-example to the coarse Baum-Connes conjecture was constructed, using expanders \cite{higson-lafforgue-skandalis}.

\begin{remark}\label{rem:Kset}
The Kazhdan set $Q$ in Theorem \ref{theorem: expanders theorem in the text} does not necessarily generate $G$. Examples of Kazhdan sets that are not generating exist already for groups $G$ and collections ${\mathcal{N}}$ having the classical property  $(\tau )$. For instance, if   $G=SL_n (\ZZ )$, a finite symmetric set generating a subgroup Zariski dense in $SL_n (\R )$ is a Kazhdan set for an appropriate choice of ${\mathcal{N}}$ (see \cite{bourgain-varju} and references therein). See also \cite{shalom-expander} for an earlier example of non-generating Kazhdan set for actions on expanders that are finite quotients. 
\end{remark}

\begin{remark}
A consequence of Theorem \ref{thm:introExp} is a modified estimate of rapid mixing of random walks for expanders coming from groups with property (T). Indeed, for a regular graph $\Gamma$ on $n$ vertices such that the normalized adjacency matrix $\widehat{A}$ has the second and the last eigenvalues at most $\alpha$ in absolute value, it is known \cite[Theorem 3.2]{HooryLinialW} that for every positive integer $k$,
  \begin{equation}
  \| \widehat{A}^k \mathbf{p} - \mathbf{u}\|_1 \leq \sqrt{n} \alpha^k, 
\end{equation} where $\mathbf{p}\in {\mathbb{R}}_+^n$ is an arbitrary probability distribution (i.e. the sum of its non-negative coordinates is $1$) and $\mathbf{u}$ is the uniform probability distribution (i.e. all its coordinates are equal). In other words, the random walk converges in $\ell_1$ to the uniform distribution exponentially fast with base $\alpha$, but the speed is slowed down by a multiplicative factor of square of the number of vertices.

For expanders coming from finite quotients of a group $G$ with property (T), the multiplicative factor in $n$ can be improved. Indeed, it follows from Remark \ref{rmks:kazhdanlp}, \eqref{lp1}, from Theorem \ref{theorem in intro: spectral gaps and projections}, and from H\"older's inequality that given a finite set $S$ generating $G$, for every $\epsilon \in \left( 0, \frac{1}{2} \right)$ there exists $\alpha =\alpha (\epsilon , G)$ such that on any Cayley graph $\Gamma$ of a finite quotient $G/N$ with respect to the image of $S$, the rapid mixing on $\Gamma$ can be estimated as follows:
 \begin{equation}\label{eq:alpha}
  \| \widehat{A}^k \mathbf{p} - \mathbf{u}\|_1 \leq n^\epsilon\, \alpha^k,\, \forall k\in {\mathbb{N}},  
\end{equation} where $n =\left|G/N \right|$, $\widehat{A}$ is the normalized adjacency matrix of $\Gamma$ and $\alpha$ is the constant provided by the effective version of Theorem \ref{theorem in intro: spectral gaps and projections}, (iii), for isometric representations of $G$ on $L^p$--spaces, with $p= \frac{1}{1-\epsilon}$. Moreover, according to Remark \ref{rem:Kset}, the same result can be obtained for random walks restricted to certain edges, i.e. for $\widehat{A}$ the normalized matrix of the adjacency only by edges labeled by elements in a proper subset $Q\subset S$, provided that $Q$ is a Kazhdan set.                    

The right choice of $\epsilon$ in \eqref{eq:alpha} is decided by the value $p\in [1,2]$ yielding the minimal value for $\alpha\in (0,1)$.  
\end{remark}

\subsection*{Applications to ergodic theory}

Another area in which Theorem \ref{theorem in intro: spectral gaps and projections} finds natural applications is ergodic theory. Consider, for instance, a group with property (T) acting ergodically on a probability space $(X, \nu )$, and let $f$ be an arbitrary function in $L^2(X)$. If the two operators appearing in the equality \eqref{formula:Neumannser} are applied to $f$, then the left hand side becomes $\int_X f \, {\mathrm{d}}\nu$, and the entire formula becomes a von Neumann-type theorem, in which an exact explicit formula is provided, instead of just an estimate for the remainder.

Thus, while for ergodic actions of ame\-na\-ble groups the best way to average is 
{\textit{via}} sequences of F\o lner sets, for ergodic actions of groups with property $(T)$ a most effective averaging is {\textit{via}} sequences of measures with compact support approximating the Kazhdan projection.

\comment
Generalizations to group actions of ergodic theorems of Birkhoff and von Neumann have been an object of significant interest, see \cite{bufetov-klimenko,nevo}
for a survey and  history of this subject. An early result of this type  is the classical ergodic theorem of Oseledec \cite{oseledets}, 
that the time averages of a function over convolution powers of 
a probability measure converge to the the mean of the function over the space.
Theorem \ref{theorem in intro: spectral gaps and projections} allows to achieve a much stronger type of convergence, in norm topology instead of the strong operator topology,
and with uniform estimates, depending on the spectral gap. It also allows to average using measures with finite support (equal to a Kazhdan set) even in the case of non-discrete groups (see Section \ref{sec:Kazhdansets}). 
\setcounter{ergthm}{\value{theorem}}
\begin{theorem}\label{t2}
Let $G$ be a locally compact group and let $(X_i, \nu_i), i\in I,$ be a family of probability spaces endowed with measure preserving ergodic actions of $G$. Consider also a collection $\mathcal E$ of uniformly convex  Banach spaces, and a number $p$ in $(1 , \infty )$. 

Assume that a family $\mathcal{F}$ of isometric representations of $G$ on $L_p(X_i,\nu;E)$, with $i\in I$ and $E\in {\mathcal E}$, induced by the measure preserving actions of $G$, has a spectral gap and let $(Q, \kappa )$ be a Kazhdan pair. 

For every $Q$--admissible measure $\mu$  on $G$ there exists $\lambda<1$, depending only on $p$, the normalizing factor of $\mu$, the modulus of convexity of $\mathcal E$, and $\kappa$, such that 
\begin{equation}\label{equation: convergence mean ergodic thm for a class F}
\left\Vert A_\pi^{\mu^k} f - \int_X f \, {\mathrm{d}}\nu\right\Vert_p \le \lambda^k \Vert f\Vert_p.
\end{equation}
\end{theorem}
\begin{proof}
Since $\mathcal{P}_\pi = \int_X  f \, d\nu$, the assertion follows from Theorem \ref{theorem in intro: spectral gaps and projections} (or, more precisely, from Theorem \ref{theorem: spectral gap to markov<1 estimate}).
\end{proof}

Theorem \ref{t2} becomes very concrete in certain cases. When $G$ has property $(T)$ and $E=\RR$, we obtain a uniform quantitative ergodic theorem for \emph{all} probability preserving actions of $G$ and for any fixed $1<p<\infty$, 
see Theorem \ref{theorem: quantitative mean ergodic theorem for Kazhdan groups}.
Another case is when $G=SL_3(F)$ for $F$ a non-archimedean local field. It follows from \cite{lafforgue-duke} that for any uniformly convex $E$
and any $1<p<\infty$ the family of isometric representations of $G$ on $L_p(X,\nu;E)$ has a spectral gap. Consequently, the above theorem holds for such $G$ 
and any fixed $E$ and $p$ as above. 
\endcomment 

Ergodic theory is a natural setting in which Theorem \ref{theorem in intro: spectral gaps and projections} can be applied. We use the theorem to investigate \emph{shrinking target problems}, which ask how often does a typical orbit of an action hit a sequence of shrinking subsets. This problem for orbits of cyclic (or uniparameter) groups in locally symmetric spaces, and for shrinking sequences of neighborhoods of a cusp, has been thoroughly investigated and answered in \cite{sullivan:log , kleinbock-margulis:log}. In the latter paper was formulated the problem of finding similar results for shrinking sequences of neighborhoods of a point. The case of rank 1 locally symmetric spaces has been settled \cite{sullivan:log,maucourant}, but the problem remains open in the case of higher rank. Theorem 
\ref{theorem in intro: spectral gaps and projections} allows to provide quantitative estimates in terms of random walks for the behavior of an ergodic action of a group with property $(T)$ with respect to a shrinking target.
For instance, we have the following theorem (see Section \ref{sec:ergodic}).

\begin{theorem}[see Theorem \ref{thm:RW}]\label{theorem: shrinking targets intro statement}
Let $G$ be a locally compact group, and $\Gamma $ a lattice in it. Let $\set{\Omega_n}$ be a sequence of measurable subsets in $G/\Gamma$. 

Assume that a locally compact group $\Lambda$ with property $(T)$ acts ergodically on $G/\Gamma$. Let $\mu$ be a probability measure on $\Lambda$ 
admissible with respect to some Kazhdan set, and let $X_n$ be the random variable representing the $n$-th step of the random walk defined by $\mu$.
\begin{enumerate}
\item If $\sum_n \nu (\Omega _n)$ is finite then for almost every $x\in G/\Gamma$
$$\sum_{n\in \N }\mathbb{P} \left( X_n (x) \in \Omega_n \right) < \infty .$$
\item If $\sum_n \nu (\Omega _n)$ is infinite then for every $\varepsilon >0$ and 
for almost every $x\in G/\Gamma$, 
\begin{equation}\label{eq:proba-nonsum}
\sum_{n\le N}{\mathbb{P}} \left( X_n (x) \in \Omega_n \right) = S_N + O\left( S_N^{\varepsilon} \right) \, , 
\end{equation} where
 $S_N = \sum_{n\le N} \nu (\Omega _n)$.
In particular, for infinitely many $n\in \N$, 
$$
{\mathbb{P}} \left( X_n(x) \in \Omega_n \right) > \nu (\Omega _n) - O\left( \nu ( \Omega _n )^{\varepsilon} \right)  . 
$$  
\end{enumerate}
\end{theorem}

Note that the error term estimate above, $O\left( S_N^{\varepsilon} \right)$, can only be obtained, as far as we know, using the fact that property (T) is equivalent to the similar property for unitary actions on $L^p$--spaces (see Remark 
\ref{rmks:kazhdanlp}), \eqref{lp1}) and the equivalence in Theorem 
\ref{theorem in intro: spectral gaps and projections} for actions on $L^p$--spaces. 

\medskip

Moreover, when $\Lambda $ is endowed with a word metric $\dist_\Lambda$ corresponding to an arbitrary compact generating set, in the theorem above one may obtain an estimate similar to the one in \eqref{eq:proba-nonsum} for the smaller probabilities 
$$
{\mathbb{P}} \left( X_n (x) \in \Omega_n , \dist_\Lambda (X_n, e)\geq an  \right)\,, 
$$ where $a>0$ is a constant depending on the choice of the word metric and of~$\mu$. 

These results apply for instance when $G$ is a semisimple group, $\Gamma$ a lattice in $G$ and $\Lambda $ an infinite subgroup of $G$, or when $G/\Gamma$ is the $n$-dimensional torus and $\Lambda $ is a subgroup of $SL_n(\Z )$.

Theorem \ref{theorem: shrinking targets intro statement} has several applications, explained in Section \ref{sec:ergodic}. We mention here only one of them. Consider the symmetric space $\calp_s= SO(s)\backslash G$ and the locally symmetric space $\calp_s/SL(s,\Z )$. Let $D$ be the common dimension of $\calp_s$ and $\calp_s/SL(s,\Z )$, and let $x$ be an arbitrary point in the latter space. For a fixed slope that is maximal singular (see the end of Section \ref{sec:ergodic} for a definition) there exists a constant $a>0$ such that the following holds. Almost every horosphere $\mathcal H$ in $\calp_s$ with point at infinity of the fixed given slope, and almost every point $h$ on $\mathcal H$, have the property that the projection of the annulus ${\mathcal{H}} \cap \left[ B(h, n)\setminus B(h, an) \right]$ in $\calp_s/SL(n,\Z )$ intersects the shrinking ball $ B(x, n^{1/D}) $ for infinitely many $n$. The same holds if one replaces the condition above with the one that the projection of ${\mathcal{H}} \cap \left[ B(h, n)\setminus B(h, an) \right]$ in $\calp_s/SL(n,\Z )$ rises infinitely many times into the cusp at a height between $\eta \ln n$ and $\eta \ln n +\frac{1}{n^\delta}$, where the height in the cusp is measured by a certain fixed Busemann function, $\delta $ can be any number in $(0,1)$, and the constant $\eta$ depends on the choice of $\delta$ and of the Busemann function measuring the height.

For further details and applications we refer to the end of Section \ref{sec:ergodic}.

\subsection*{Obstructions to the coarse Baum-Connes conjecture}

The final application we present concerns obstructions to the coarse Baum-Connes conjecture. In \cite{higson, higson-lafforgue-skandalis} it was shown that 
the coarse Baum-Connes conjecture fails for coarse disjoint unions of expander graphs arising from an exact group with property $(T)$. The reason is the 
existence of a certain Kazhdan-type projection, a non-compact ghost projection, which is a limit of finite propagation operators. Until now such ghost projections
were constructed only for expanders. Willett and Yu  formulated the following

\begin{problem}[{\cite[Problem 5.4]{willett-yu}}]
Find other geometric examples of ghost projections.
\end{problem}

Motivated by their question we construct non-compact ghost projections for warped cones \cite{roe-warped}.

Let $G$ be a finitely generated group acting ergodically by probability preserving Lipschitz homeomorphisms on a compact metric probability 
measure space $(M,\dist,m)$.
Assume that the measure $m$ is upper uniform, i.e. it is distributed uniformly over $M$ with respect to the metric, see Definition \ref{definition : upper uniform measure}.
Denote by $\mathcal{O}=\mathcal{O}_G(M)$ the warped cone associated to the action of $G$ on $M$, as defined in Section \ref{sect:warped} (see also \cite{roe-warped}).
\setcounter{warpthm}{\value{theorem}}
\begin{theorem}[see Theorem \ref{theorem: non-compact ghost projection}]\label{theorem: intro warped cones}
If the action of $G$ on $(M,m)$ has a spectral gap then there exists a non-compact ghost projection $\mathfrak{G}\in \mathcal{B}(L_2(\mathcal{O}))$ which 
is a limit of finite propagation operators.
\end{theorem}
We also conjecture that such warped cones with non-compact ghost projections as provided by Theorem \ref{theorem: intro warped cones} are counterexamples to the 
coarse Baum-Connes conjecture. Warped cones in many ways exhibit a behavior similar to that of box spaces, however one expects that the 
fact that the class of warped cones is so rich might lead to examples with interesting new features.

\subsection*{Acknowledgments}
The second author would like to thank the Mathematical Institute at the University of Oxford for its hospitality during a 4-month stay 
which made this work possible. Both authors thank Mikael de la Salle, Adam Skalski, Alain Valette, Rufus Willett and Guoliang Yu for valuable comments.
We are grateful to the referee for numerous suggestions that allowed to improve the presentation significantly.

\tableofcontents

\section{Preliminaries}

\subsection{Uniform convexity}
Let $(E,\Vert \cdot \Vert_E)$ be a reflexive Banach space, and let ${\mathcal{B}}(E)$ be the algebra of bounded linear operators on $E$.
The {\emph{modulus of convexity of}} $E$ is the function $\delta_E:[0,2]\to [0,1]$ defined by
$$\delta_E(t)=\inf \left\{ 1-\left\Vert \dfrac{v+w}{2}\right\Vert \; ;\; \Vert v\Vert=\Vert w\Vert=1, \Vert v-w\Vert\ge t\right\}.$$
The Banach space $E$ is said to be {\emph{uniformly convex}} if $\delta_E(t)>0$ for every $t>0$.

A family $\mathcal{E}=\set{E_i}_{i\in I}$ of Banach spaces is \emph{uniformly convex} if the function $\delta_{\mathcal{E}}(t)=\inf_{i\in I} \delta_{E_i}(t)$, called the {\emph{modulus of convexity of the family}} $\mathcal{E}$, satisfies $\delta_{\mathcal{E}}(t) >0$ for every $t>0$.

\subsection{Admissible measures }\label{subsection: measures}

Compactly supported pro\-ba\-bility measures on topological groups and the corresponding random walks are central objects in our arguments. We introduce some notation and several standing assumptions on such measures. 

Consider $G$ a locally compact group, endowed with a (left invariant) Haar measure $\eta$. 
For any function $f:G\to {\mathbb{C}}$ we denote $\gamma\cdot f(g)=f(\gamma^{-1} g)$, $\gamma,g\in G$.  

We consider two particular cases, before introducing the notion of admissible measure in full generality. Let $Q$ be a compact subset of $G$.

\subsubsection*{Continuous admissible measures}
Let $\alpha,\beta:G\to [0,\infty)$ be continuous functions with compact support satisfying 
$$\int \alpha\,d\eta=1 \ \ \ \ \ \text{ and } \ \ \ \ \ \ \beta(g)\ge s\cdot \alpha(g), \ \ \ \ \ \forall s\in  Q, g\in G\, .$$
We also assume that $\alpha(e)>0$.
Another continuous function, whose compact support contains $ Q$, can then be defined by the formula
\begin{equation}\label{equation: alpha-beta-decomposition of rho}
\rho=\dfrac{\alpha+\beta}{M(\alpha,\beta)}, \ \ \ \ \ \text{ where } \ \ \ \ \ \ M(\alpha,\beta)=\int_G (\alpha+\beta) \,d\eta.
\end{equation}
We call  a decomposition as in \eqref{equation: alpha-beta-decomposition of rho} an \emph{$(\alpha,\beta)$-decomposition of $\rho$}.

The function $\rho$ gives rise to a probability measure $\mu$ on $G$ defined by setting
$d\mu = \rho\, d \eta$.

\subsubsection*{Discrete admissible measures} Now consider functions with finite support $\alpha, \beta:G\to [0,\infty)$. 
With the above conditions formulated for such $\alpha$ and $\beta$, we define $\rho$ as in \eqref{equation: alpha-beta-decomposition of rho},
where $M(\alpha,\beta)=\sum_{g\in \supp \alpha\cup \supp \beta} \left[ \alpha(g)+\beta(g)\right],$ and the formula \eqref{equation: alpha-beta-decomposition of rho}  is again called an \emph{$(\alpha,\beta)$-decomposition of $\rho$.}

As $\rho$ has finite support and $\sum_{g\in\supp\rho} \rho(g)=1$, it gives rise to a purely atomic probability measure on $G$.

\begin{definition}\label{def:admiss}
A measure $\mu_c$ (respectively $\mu_d$) will be called a \emph{continuous admissible measure with respect to $ Q$} 
(respectively,  \emph{discrete admissible measure with respect to $ Q$}) 
if it is defined by a continuous density $\rho$ (respectively, a finitely supported function $\rho$) admitting an $(\alpha,\beta)$-decomposition.

A probability measure $\mu$ on $G$ will be called \emph{admissible with respect to $ Q$} if there exists $t\in[0,1]$ such that $\mu= t\mu_c+(1-t)\mu_d$, where 
$\mu_c$ and $\mu_d$ are, respectively, continuous and discrete admissible measures with respect to $Q$.

The {\emph{normalizing factor}} of the function $\rho$ and its associated continuous or discrete measure  $\mu$ is the infimum of $M(\alpha,\beta)$, taken
over all $(\alpha,\beta)$-decompositions of $\rho$, with $Q$ fixed. This factor will be  denoted either $M_\rho$ or $M_\mu$, depending on the object referred to. 

The {\emph{normalizing factor}} of an admissible measure  $\mu=t\mu_c+(1-t)\mu_d$ is the number 
$$M_\mu=\left(\dfrac{t}{M_{\mu_c}}+ \dfrac{1-t}{M_{\mu_d}}\right)^{-1}.$$
\end{definition}

Admissible measures always exist on a locally compact group. We expect that the arguments presented in this paper would, 
with some modifications, work for a larger class of measures. 
However, the above setting allows to identify a continuous admissible measure $\mu$ naturally, {\textit{via}} the density $\rho$, with an element of the group algebra 
$C_c(G)$ (respectively the group ring $\CC G$, in the discrete 
case), which is crucial for further applications.

The set $ Q$ will usually be a Kazhdan set (see Definition \ref{definition: spectral gaps}). Since such sets can be finite even for Lie groups (see Section \ref{sec:Kazhdansets}), it 
is useful to work with measures having an atomic part even in the Lie group setting. When $\mu$ is continuous, hence entirely defined by a density $\rho$ with respect to the Haar measure $\eta$, we sporadically replace $\mu$ by $\rho$ in the whole notation and terminology.

\medskip

The following was pointed out to us by the referee.

\begin{remark}\label{rmk:referee}\normalfont
The definition of admissible measures can be motivated by considering the case of Hilbert spaces and unitary representations of a finitely generated group $G$. 
Note that, given an element $u\in \CC G$ defined as the uniform probability measure on a finite symmetric generating set $S$, as in \cite{harpe-robertson-valette,valette}, 
the associated random walk
does not necessarily converge to the Kazhdan projection even if 1 is an isolated point in the spectrum of $u$.
However, it does converge if $-1$ does not belong to the spectrum of $u$.
The latter condition is equivalent to the property that the Cayley graph of $G$ with respect to $S$ is not bi-colorable \cite[Proposition II]{harpe-robertson-valette}.
For instance, if $G$ has property $(T)$, the Cayley graph of $G\times \ZZ/2\ZZ$ is bi-colorable, while the Cayley graph of $G$ itself is not bi-colorable, provided that $G$ does not have non-trivial finite quotients.
The spectrum of a Markov operator associated to an admissible measure automatically does not contain $-1$, which in the Hilbert space case 
explains from the point of view of spectral theory the convergence of the 
associated random walk to the Kazhdan projection and provides additional motivation for using admissible measures.
\end{remark}

\subsection{Groups and representations}\label{subsection: groups and representations}
Let $G$ be a locally compact group.
An isometric representation $\pi :G\to \mathcal{B}(E)$ of $G$ on a Banach space $E$ is said to be {\emph{continuous}} if it is continuous with respect to 
the strong operator topology. Equivalently, every orbit map is continuous, see \cite[Lemma 2.4]{bfgm}.
Throughout the article we restrict our attention to representations that are continuous in the above sense, without mentioning this further.


Consider the subspace of $E$ consisting of vectors invariant under $\pi$,
 $$E^\pi=\setd{ v\in E}{\pi_g v=v \text{ for every } g \in G}.$$

The dual space $E^*$ is naturally equipped with a contragradient  representation $\opi:G\to \mathcal{B}(E^*)$, defined by the formula
 $\opi_g = \pi_{g^{-1}}^*$.
Note that $\opi$ is isometric if $\pi$ is, but not necessarily continuous.
If $E$ is reflexive we define a subspace
$E_\pi=\operatorname{Ann}\left((E^*)^{\opi}\right)$, 
where $\operatorname{Ann}$ denotes the annihilator: the set of all functionals in $E=E^{**}$ that vanish identically on $(E^*)^{\opi}$.
Both $E^\pi$ and $E_\pi$ are $\pi$--invariant closed subspaces of $E$.
 
\begin{definition}\label{def:compl}
A representation $\pi :G \to {\mathcal{B}}(E)$ is {\emph{complemented}} if 
\begin{equation}\label{equation: decomposition into invariant vectors + complement}
E=E^\pi \oplus E_\pi.
\end{equation}
A family of complemented representations is called a {\emph{complemented family}}.
\end{definition}

Examples of complemented representations include
 isometric representations on reflexive Banach spaces \cite{bader-rosendal-sauer} (in particular, on uniformly convex Banach spaces \cite[Section 2.c]{bfgm}),
and representations of small exponential growth of certain Lie groups on Banach spaces with non-trivial Rademacher type \cite{lafforgue-duke}.

A representation $\pi:G\to {\mathcal{B}}(E)$ is \emph{uniformly bounded}  if
$\Vert \pi \Vert=\sup_{g\in G}\Vert \pi_g\Vert_{{\mathcal{B}} (E)}<\infty$.
For any such representation $\pi$, a new norm can be defined on $E$, equivalent to the initial one, by the formula
\begin{equation}
\label{equation: invariant norm for ub reps}
\Vert v\Vert_\pi=\sup_{g\in G } \Vert \pi_g v\Vert.
\end{equation}
As observed in \cite[Proposition 2.3]{bfgm}, the modulus of convexity of the norm $\Vert \cdot \Vert_\pi$
satisfies $\delta_{\Vert\cdot\Vert_\pi}(t)\ge \delta_{\Vert \cdot\Vert}\left(t\Vert \pi\Vert^{-1} \right)$ for every  $t>0$.
\subsection{Spectral gaps and uniform property $(TE)$}

Throughout the section, $G$ is a locally compact group and $E$ a Banach space.

A representation $\pi$ of $G$ on $E$ has \emph{almost invariant vectors} if for every $\varepsilon>0$ and every compact subset $S$ in $G$ there exists $v\in E$, $\Vert v\Vert=1$,
such that
$\sup_{s\in S} \Vert v-\pi_s v \Vert \le \varepsilon$.

\begin{definition}\label{definition: spectral gaps}
\begin{enumerate}
\item A complemented representation $\pi: G \to \mathcal{B}(E)$ has a \emph{spectral gap} 
if the restriction of $\pi$ to $E_\pi$ does not have almost invariant vectors, i.e. if there exists a constant $\kappa >0$ and a compact subset $Q$ in $G$ such that for every $v\in E_{\mathcal{\pi}}$
$$\sup_{s\in Q} \Vert v-\pi_sv\Vert\ge c\Vert v\Vert \, .$$
Any such pair $(Q, \kappa )$ is called a \emph{Kazhdan pair for }$\pi$, $\kappa$ is called \emph{a Kazhdan constant}, and $Q$ \emph{a Kazhdan set}.\\
\item 
We say that a complemented family $\mathcal{F}$ of representations on a family of Banach spaces \emph{has a  spectral gap} if there exist  $Q$ and $\kappa>0$ as above such that $(Q,\kappa)$ is a Kazhdan pair for every $\pi\in\mathcal{F}$. 
 We call $(Q, \kappa )$  a \emph{Kazhdan pair for }$\mathcal F$. 
\end{enumerate}
\end{definition}

If $\mathcal{F}$ is a family of representations closed under direct sums then the existence of a spectral gap for each $\pi$ in $\mathcal{F}$ is equivalent to the existence of a spectral gap for the entire family.

In the particular case when $\mathcal{F}$ is composed of all the unitary representations of a group $G$, the Kazhdan constant in the sense of the above definition 
is the classical Kazhdan constant associated to a Kazhdan set, see \cite{bhv}. 
In that case, every generating set of $G$ is a Kazhdan set, but the converse is true only for sets with non-empty interior. See Section \ref{sec:Kazhdansets}.

The following definition introduces versions of Kazhdan property $(T)$ for Banach spaces.
\begin{definition}\label{TEU}
Let $\mathcal{E}$ be a family of Banach spaces and let $E\in\mathcal{E}$.
\begin{enumerate}
\item $G$ \emph{has property $(TE)$} if each isometric representation $\pi$ of $G$ on $E$ has a spectral gap \cite{bfgm}. More generally, 
$G$ has \emph{property $(T\mathcal{E})$} if 
every isometric representation of $G$ on any $E\in \mathcal{E}$ has a spectral gap.
\item $G$ has \emph{property $(TE)$ uniformly}
if  the family of all isometric representations  of $G$ on $E$ has a spectral gap. More generally, $G$ has \emph{property $(T\mathcal{E})$ uniformly} 
if the family of all isometric representations of $G$ on all Banach spaces 
$E\in \mathcal{E}$ has a spectral gap.
\end{enumerate}
\end{definition}

\begin{remarks}\label{rmks:kazhdanlp}\normalfont
\begin{enumerate}
\item\label{lp1} It was proved in \cite{bfgm} that if $G$ has property $(T)$ then it has property $(TE)$ for  $E=L_p(X,\nu)$, $1\leq p<\infty$. Moreover, there exists a Kazhdan pair $(Q, \kappa )$ common to all isometric representations on $L_p(X,\nu)$, with $1\leq p \leq 2$ and such that for $p>2$, $\left( Q, \frac{2}{p}\kappa \right)$ is a Kazhdan pair for isometric representations on $L_p(X,\nu)$. The former follows from standard embedding results \cite{wellswilliams}, while the latter can be deduced from estimates of the Mazur map \cite[page 198]{bl}.

\medskip

\item\label{lp2} In \cite{bekka-olivier}  property $(T\ell_p)$ was studied systematically. It yields a larger class than that of groups with property $(T)$, 
containing for instance irreducible lattices in products of locally compact second countable groups, one with property $(T)$ 
and the other with no non-trivial finite dimensional unitary representation. 
\end{enumerate}
\end{remarks}  


\section{Random walks, projections and spectral gaps}

\subsection{The Markov operator $A_\pi^\mu$ and its properties}\label{subsec:averageop}

Let $\pi :G \to {\mathcal{B}}(E)$ be a uniformly bounded representation of $G$ on a Banach space $E$, and let $\mu$ be a probability measure 
on $G$. The operator $A^\mu_\pi :E \to E$, defined by the Bochner integral
$$A^\mu_\pi v = \int_G \pi_g v  \, d\mu (g)$$
 is called the {\emph{Markov operator}} associated to the random walk on $G$ determined by $\mu$. 
By standard properties of the Bochner integral, we have that $ A_\pi^\mu$ is a bounded operator as well.

The operator $A^\mu_\pi$ can also be defined for bounded representations that are not uniformly bounded, provided that the support of $\mu $ is compact.

\subsubsection*{General properties of $A_\pi^\mu$}
In the following lemma we collect several standard  properties that we will need later, of the operator $A_\pi^\mu$ for an isometric representation $\pi$. Denote by $\overline{\mu }$ the measure obtained from $\mu$ by pre-composing with the map $s\mapsto s^{-1}$ in $G$ and post-composing with the conjugation in $\mathbb C$.

\begin{lemma}\label{lemma: A_pi preserves the decomposition into E^pi and E_pi}
Let $\mu$ be a probability measure on $G$. Then\begin{enumerate}
\item $\left(A_\pi^\mu\right)^*=A_{\opi}^{\overline{\mu}}$\,,\label{lemma claim: conjugate of A}
\item $A_\pi^{\mu*\nu}=A_\pi^\mu A_\pi^\nu$,
\item $\pi_g A_\pi^\mu =A_\pi^{g\cdot\mu} $  for every $g\in G$,\label{lemma claim: acting by g on A}
\item $A_\pi^\mu=I$ on $E^\pi$, \label{lemma claim: A_pi = I on E^pi}
\item $A_\pi^\mu (E_\pi)\subseteq E_\pi$.
\end{enumerate}
\end{lemma}

\subsubsection*{Properties of $A_\pi^\mu$ with respect to a lattice}

Consider now a locally compact group $G$ with a finitely generated lattice $\Gamma$, i.e. a finitely generated subgroup $\Gamma$ such that 
$G/\Gamma$ has a finite $G$--invariant measure induced by the Haar measure.

Let $\pi$ be a continuous isometric representation of $G$ on a reflexive Banach space $E$. Denote by $\pi\vert \Gamma$ the restriction of $\pi$ to the lattice $\Gamma$.
The inclusion of $\Gamma$ into $G$ gives rise to two decompositions,
\begin{equation}\label{equation: inclusions of subspaces of E for a lattice} 
E=E^\pi\oplus E_\pi=E^{\pi\vert \Gamma}\oplus E_{\pi\vert\Gamma}.
\end{equation}
Since $\overline{\pi}\vert\Gamma=\overline{\pi\vert\Gamma}$, the subspaces above satisfy $E^\pi\subseteq E^{\pi\vert\Gamma}$ 
and $E_{\pi\vert\Gamma}\subseteq E_\pi$.

\begin{lemma}\label{lemma: decomposition E_G into E_Gamma and complement}
There is a direct sum decomposition
$E_\pi=E_{\pi\vert \Gamma}\oplus \left( E_\pi\cap E^{\pi\vert\Gamma} \right)$.
\end{lemma}
\begin{proof}
Let $v\in E_\pi$. Then $v=w+z$, where $w\in E_{\pi\vert\Gamma}$ and $z\in E^{\pi\vert\Gamma}$. Thus 
$z=v-w\in E_{\pi}$ by \eqref{equation: inclusions of subspaces of E for a lattice}.
\end{proof}
The above decomposition is preserved by $\pi\vert\Gamma$ but in general not preserved by $\pi$.

Now choose a fundamental domain $\Delta$ for $\Gamma$ in $G$ and renormalize the Haar measure $\eta$ on $G$ so that $\eta(\Delta)=1$.
For the purposes of the next statement denote by $A_\pi^\Delta$ the Markov operator associated to the measure determined by the  (possibly discontinuous) characteristic function of $\Delta$ as a density function.
The next proposition shows that the restriction of $A_\pi^\Delta$ to $E_\pi$ is concentrated on $E_{\pi\vert\Gamma}$.

\begin{proposition}\label{proposition: norm of Markov operator concentrated on the lattice}
Let $v\in E_\pi\cap E^{\pi\vert \Gamma}$. Then 
$A_\pi^\Delta v=0$.

In particular, given $v\in E_\pi$ we have $A_\pi^\Delta v=A_\pi^\Delta w$, where $w\in E_{\pi\vert\Gamma}$ is as in the previous lemma.
\end{proposition}
\begin{proof}
Observe that $A^\Delta_\pi v \in E^\pi$. Indeed, for any $h \in G$,
\begin{align*}
\pi_hA^\Delta_\pi v &= \int_\Delta  \pi_{hg}v\,d\eta(g)= \int_{h\Delta}  \pi_{g}v\,d\eta(g)\\
&= \sum_{\gamma \in \Gamma} \int_{h\Delta\cap \Delta \gamma }  \pi_{g}v\,d\eta(g)=\sum_{\gamma\in\Gamma} \int_{h\Delta\gamma^{-1}\cap \Delta} \pi_{g\gamma } v d\eta(g)
\end{align*}
where in the last equality the change of variable $g\mapsto g\gamma^{-1}$ and the unimodularity of $G$ (due to the existence of a lattice) were used.

The above and the fact that $v$ is fixed by $\Gamma$ imply that
$$
\pi_hA^\Delta_\pi v= \sum_{\gamma\in\Gamma} \int_{h\Delta\gamma^{-1}\cap \Delta} \pi_g v d\eta(g)=  \int_\Delta \pi_gv\,d\eta(g)= A^\Delta_\pi v\, .
$$
Since $A^\Delta_\pi v\in E_\pi\cap E^\pi$, the assertion follows.
\end{proof}

\subsection{Proof of Theorem \ref{theorem in intro: spectral gaps and projections}}

Our central result establishes a connection between Kazhdan constants, convergence of iterated Markov operators, and projections onto the subspace of invariant vectors.

\setcounter{aux1}{\value{section}}
\setcounter{aux2}{\value{theorem}}
\setcounter{section}{1}
\setcounter{theorem}{\value{mainthm}}

\begin{theorem}\label{theorem: main theorem in text spectral gaps and projections}
Let $G$ be a locally compact group, and $\mathcal{F}$ a family of isometric representations of $G$ on a  uniformly convex family $\mathcal{E}$ of Banach spaces. 
The following conditions are equivalent:
\begin{enumerate} \item the family $\mathcal{F}$ has a spectral gap;\label{item: main theorem condition spectral gap}
\item there exists a compactly supported probability measure $\mu$ on $G$ and $\lambda<1$  \label{item: main theorem condition norm of Markov}
such that for every isometric representation $\pi \in \mathcal{F}$ of $G$ on $E\in\mathcal{E}$ we have $\left\lVert A_\pi^\mu\vert_{E_\pi}\right\rVert<\lambda$;
\item there exists a compactly supported probability measure $\mu$ on $G$ and a number $\mathfrak{S}<\infty$ such that for every $\pi\in \mathcal{F}$ 
the iterated Markov operators $\left(A_\pi^\mu\right)^k$ 
converge with speed summable to at most $\mathfrak{S}$ to the projection $\mathcal{P}_\pi$ onto $E^\pi$ along $E_\pi$. \label{item: main theorem condition convergence}
\end{enumerate}
\end{theorem}

\setcounter{section}{\value{aux1}}
\setcounter{theorem}{\value{aux2}}

Detailed statements and proofs of the implications composing Theorem \ref{theorem: main theorem in text spectral gaps and projections} appear in the next three sections: \eqref{item: main theorem condition spectral gap} 
$\Longrightarrow$ \eqref{item: main theorem condition norm of Markov} in section 
\ref{section: Kazhdan constants to Markov operators}; \eqref{item: main theorem condition norm of Markov} $\Longrightarrow$ \eqref{item: main theorem condition convergence} in section \ref{section: from Markov to projections};
and \eqref{item: main theorem condition convergence} $\Longrightarrow$ 
\eqref{item: main theorem condition spectral gap} in section \ref{section: from projections to Kazhdan constants}.

Note that even though the above theorem is stated for isometric representations, it is automatically true for classes of uniformly bounded representations
with a common upper bound on norms. This follows from the renormings associated to uniformly bounded representations \eqref{equation: invariant norm for ub reps}.

\subsection{From Kazhdan pairs to contracting Markov operators}\label{section: Kazhdan constants to Markov operators}
In this section, given a Kazhdan pair, we provide a construction of a Markov operator with an effective 
estimate on the norm on the subspace $E_\pi$. 
The proof we provide relies on uniform convexity.

\begin{theorem}\label{theorem: spectral gap to markov<1 estimate}
Let $G$ be a locally compact group and $\mathcal{F}$ a family of isometric representations of $G$ on a uniformly convex family $\mathcal{E}$ 
of Banach spaces. Assume that $\mathcal{F}$ has a spectral gap and let  $(Q, \kappa )$ be a Kazhdan pair for $\mathcal{F}$. 

For every $Q $--admissible measure $\mu$ on $G$, and for every isometric representation $\pi\in\mathcal{F}$ we have 
\begin{equation}\label{equation: estimate on the norm of the Markov operator}
\left\lVert A_\pi^\mu\vert_{E_\pi}\right\rVert  \le 1-\dfrac{2}{M_\mu} \delta_{\mathcal{E}}\left(\kappa \right).
\end{equation}
\end{theorem}

\begin{proof}
First assume that $\mu$ is a continuous admissible probability measure on $G$.
Let $\pi$ be an isometric representation of $G$ on $E$ with a spectral gap. 
Since $\mu$ is admissible we can choose  an $(\alpha,\beta)$-decomposition for the density $\rho$ defining $\mu$, as in Section \ref{subsection: measures}. 
We will use $\alpha$ as the upper subscript in reference to the Markov operator associated to the measure determined by the density  $\alpha$.

Let $v\in E_\pi$ be a unit vector. By Lemma \ref{lemma: A_pi preserves the decomposition into E^pi and E_pi}, 
$A_\pi^\alpha v\in E_\pi$.
Fix $s\in Q$ such that 
$$\left\lVert A_\pi^\alpha v-\pi_s A_\pi^\alpha v\right\rVert\ge \kappa \left\lVert A_\pi^\alpha v\right\rVert.$$
We have
\begin{equation}\label{eq:uconvC}
A_\pi^\mu v=\left(A_\pi^\mu v-  \left( \dfrac{A_\pi^\alpha v +\ A_\pi^{s\cdot \alpha} v}{M(\alpha,\beta)} \right)\right) +  
\left( \dfrac{A_\pi^\alpha v +\pi_s A_\pi^\alpha v}{M(\alpha,\beta)} \right).
\end{equation}

We estimate the norm of the first summand in \eqref{eq:uconvC} as follows.
\begin{align*}
\left\lVert  A_\pi^\mu v- \dfrac{1}{M(\alpha,\beta)} \left( A_\pi^\alpha v +\ A_\pi^{s\cdot \alpha} v \right)  \right\rVert &= \left\lVert \int_G \pi_g v \left(\rho-\dfrac{\alpha+s\cdot\alpha}{M(\alpha,\beta)}\right) \,d\eta\right\rVert\\
&\le  \int_G \left\lVert \pi_g v \right\rVert  \left\vert\rho-\dfrac{\alpha+s\cdot\alpha}{M(\alpha,\beta)}\right\vert \,d \eta\\
& =   \int_G \rho-\dfrac{\alpha+s\cdot\alpha}{M(\alpha,\beta)} \,d \eta\\
&= 1-\dfrac{2}{M(\alpha,\beta)},
\end{align*}
where in the last but one equality we used the fact that $\rho\ge \frac{ \alpha+s\cdot \alpha}{M(\alpha,\beta)}$, which follows from the properties of the 
$(\alpha,\beta)$-decomposition.

By the uniform convexity of $E$, the norm of the second summand in \eqref{eq:uconvC} is bounded as follows
\begin{align*}
\dfrac{1}{M(\alpha,\beta)} \left\lVert \left( A_\pi^\alpha v +\pi_s A_\pi^\alpha v \right)\right\rVert &= \dfrac{2 \left\lVert A_\pi^\alpha v \right\rVert }{M(\alpha,\beta)} 
\left\lVert \dfrac{ A_\pi^\alpha v +\pi_s A_\pi^\alpha v }
{2\left\lVert A_\pi^\alpha v\right\rVert}\right\rVert\\
&\le \dfrac{2}{M(\alpha,\beta)}\left(1-\delta_E(\kappa)\right).
\end{align*}
The two inequalities together give
\begin{align*}
\left\lVert A_\pi^\mu v\right\rVert &\le 1-\dfrac{2}{M(\alpha,\beta)} + \dfrac{2}{M(\alpha,\beta)}\left(1-\delta_E(\kappa)\right)\\
& \le 1-\dfrac{2}{M(\alpha,\beta)} \delta_E(\kappa).
\end{align*}
Passing to the infimum over all $(\alpha,\beta)$-decompositions of $\rho$ gives the estimate
$$\left\lVert A_\pi^\mu v\right\rVert \le 1-\dfrac{2}{M_\mu} \delta_E(\kappa).$$

The same calculation as above gives the same estimate when $\mu$ is a discrete admissible measure. Given any admissible 
measure $\mu=t\mu_c+(1-t)\mu_d$, $t\in [0,1]$, for $\mu_c$ and $\mu_d$ admissible continuous and admissible discrete probability measures, respectively,
we have
$$\left\lVert A_\pi^\mu\right\rVert\le t \left\lVert A_\pi^{\mu_c}\right\rVert + (1-t)\left\lVert A_\pi^{\mu_d}\right\rVert.$$
This yields the required conclusion.
\end{proof}

In the case of a finite Kazhdan set, we obtain the following.

\begin{corollary}\label{corollary: main theorem for finitely generated and uniform measure}
Let $G$ be a locally compact group and let $\mathcal{F}$ be a family of isometric representations of $G$ on a uniformly convex family $\mathcal{E}$ of Banach spaces 
such that $\kappa>0$ is a Kazhdan constant for a finite Kazhdan set $Q$, and let $g\notin Q$. Let $\mu$ be the uniform probability measure on $Qg \cup \{ g\}$. 
Then for every isometric representation $\pi\in \mathcal{F}$ we have
$$\left\lVert A_\pi^\mu\vert_{E_\pi}\right\rVert\le 1- \dfrac{2}{\# Q+1} \delta_{\mathcal{E}}(\kappa).$$
\end{corollary}

\begin{proof}
The uniform measure on $Qg \cup \{ g\}$ admits an $(\alpha,\beta)$-decomposition with 
 $\alpha$  the Dirac mass at $g\in G$ and $\beta$  the characteristic function of $Qg$. 
This decomposition also gives $M_\mu\le \#Q+1$
 and the estimate follows.
\end{proof}

\subsection{From contracting Markov operators to projections}\label{section: from Markov to projections}

Recall that the \emph{Neumann series of an operator} $T$ is the series $\sum_{n=0}^{\infty}T^n$. It is convergent if $\Vert T\Vert<1$ and
in that case it is the inverse of $I-T$. This allows to give an explicit formula for the projection onto invariant vectors in terms of the Markov operator.

\begin{theorem}\label{theorem: definition of the projection onto invariant vectors}
Let $G$ be a locally compact group and $\mu$ a probability measure on $G$.
Let $\mathcal{F}$ be a complemented family of isometric representations of $G$ on a family of Banach spaces $\mathcal{E}$.
If there exists $\lambda<1$ such that for every representation $\pi\in\mathcal{F}$ on $E\in\mathcal{E}$ 
we have $\left\| {A_{\pi }^{\mu }}|_{E_{\pi }} \right\| \le \lambda$ then for every $\pi\in \mathcal{F}$
\begin{enumerate}
\item the operator
$$\mathcal{P}_\pi = I_E- \left(\sum_{n=0}^{\infty} \left(A_{\pi }^{\mu }\right)^n \right) \left(I_E-A_\pi^\mu\right)$$
is the projection $E\to E^\pi$ along $E_\pi$ onto the subspace of invariant vectors of the representation $\pi$;
\item the iterated average operator $\left( A_{\pi }^{\mu } \right)^{k}$ converges to $\mathcal{P}_\pi$ exponentially fast, uniformly over $\mathcal{F}$,
$$\left\| \left(A_\pi^\mu\right)^k - \mathcal{P}_{\pi }  \right\|\le \lambda^k.$$
\end{enumerate}
\end{theorem}

\begin{proof} 
{\textit{(i)}.}  The operator $I-A_\pi^\mu$ is invertible on $E_\pi$ and its inverse on $E_\pi$ is given by the Neumann series
$$\left(I_{E_\pi}-A_\pi^\mu \right)^{-1} = \sum_{n=0}^{\infty} \left(A_\pi^\mu\right)^n.$$

With this in mind we proceed to show that $\mathcal{P}_\pi$ is well-defined. We observe that  for every $v\in E$ we have $\left(I_E-A_\pi^\mu\right)v \in E_\pi$.
Indeed, using the decomposition (\ref{equation: decomposition into invariant vectors + complement}) we can write $v=z+w$ uniquely, where $z\in E^\pi$ and $w\in E_\pi$.
We have
$$\left(I_E-A_\pi^\mu\right)v=w-A_\pi^\mu w \in E_\pi,$$
since $\left(I_E-A_\pi^\mu\right)z=0$, by Lemma  \ref{equation: decomposition into invariant vectors + complement}.
Since $(I-A_\pi^\mu)v \in E_\pi$ and the Neumann series of $A_\pi^\mu$ converges on $E_\pi$, we see that $\mathcal{P}_\pi$ is well-defined and bounded.

Observe also that since $\left(I-A_\pi^\mu\right)z=0$ for $z\in E^\pi$, we have $\mathcal{P}_\pi z=z.$
On the other hand, if $w\in E_\pi$ then
$$\left(\sum_{n=0}^{\infty} \left(A_\pi^\mu\right)^n \right) \left(I_E-A_\pi^\mu\right) w=w,$$
 and, consequently, $\mathcal{P}_\pi w=0$.
Therefore, given any vector $v=z+w$ as above we have $\mathcal{P}_\pi(z+w)=z.$

{\textit{(ii)}.} Observe  $\mathcal{P}_\pi$ is a norm limit of operators $\mathcal{P}_{\pi,k}$, defined by truncating the Neumann series to its $k$-partial sum, and  $\mathcal{P}_{\pi,k}=\left(A_\pi^\mu\right)^{k+1}$.
We prove the inequality by induction on $k$. For $k=1$ we can write 
$$ \left\| A^\mu_\pi  -\calp_\pi \right\| =\sup \setd{ \left\|  A^\mu_\pi w  \right\| }{w\in E_\pi, z\in E^\pi , \| z+ w \|=1 } \leq \lambda.$$

Assume that the inequality is proven for $k$. Since $A^\mu_\pi \circ \calp_\pi = \calp_\pi$ and since the image of $\left(A^\mu_\pi\right)^{k}  - \calp_\pi$ is in $E_\pi$, we can write
$$
\left\| \left(A^\mu_\pi\right)^{k+1}  - \calp_\pi \right\|= \left\| \left(A^\mu_\pi\right)^{k+1}  - A^\mu_\pi\circ \calp_\pi \right\|   
\leq \lambda \left\| \left(A^\mu_\pi\right)^{k}  - \calp_\pi \right\| \leq \lambda^{k+1}.
$$
\end{proof}

In the particular case of finite Kazhdan sets, one can replace the iteration of an average operator by products of average operators. Indeed, let $G$ be a locally compact group  and let $\mathcal{F}$ be a family of isometric representations of $G$ on a uniformly convex family of Banach spaces, $\mathcal{F}$ admitting a Kazhdan pair $(X, \kappa )$ with $X =\left\{ x_1,\dots ,x_N \right\}$. Let $S_n$ be a sequence of finite sets constructed in one of the following manners 
\begin{enumerate}\renewcommand{\theenumi}{\alph{enumi}}
\item $S_n = X\cup Y_n$, where $Y_n = \{ y_1, \dots , y_M\}$ is an arbitrary subset of $M$ elements in $G$, where $M$ is fixed;
\item $S_n = X_n \cup Y_n$, where $Y_n= \{ y_1, \dots , y_N\}$ and 
$X_n = \setd{ y_ix_iy_i^{-1}}{1\le i\le N }.$
\end{enumerate}
Let $\mu_n$ be the atomic uniform measure on the set $S_n$.

\begin{corollary}\label{cor:amufg} 
For the sequence of measures $\mu_n$ constructed  above and for every representation $\pi\in \mathcal{F}$ we have
$$\left\| A_{\pi}^{\mu_1}A_{\pi}^{\mu_2}\cdots A_{\pi}^{\mu_n} - \calp_\pi \right\|\le \left( 1-\dfrac{2}{N }\delta_E(\kappa /3) \right)^n. $$
\end{corollary}

\begin{proof}
For $S_n = X \cup Y_n$ the Kazhdan constant is at least $\kappa$. Let now $ S_n = X_n\cup Y_n$. For every $v\in E_\pi , \| v\| =1,$ there exists $x\in X$ such that $\| \pi_x v - v\| \ge \kappa$. The triangle inequality allows to write, for the corresponding element $y\in Y_n$ that
$$
\kappa \leq \| \pi_y v - v\| + \| \pi_{y^{-1}} v - v\| + \| \pi_{y x y^{-1}} v - v\|.
$$ 

It follows that at least one of the terms in the sum above is larger than $\frac{\kappa }{3}\, $. Therefore in this case the Kazhdan constant is at least $\frac{\kappa }{3}\, $. This, Corollary \ref{corollary: main theorem for finitely generated and uniform measure} and an easy induction on $n$ yields the inequality.
\end{proof}

\comment
\begin{remark}\normalfont
Recall that if $G$ is a finite group then the projection $\mathcal{P}_\pi$ is given by an explicit formula,
$\mathcal{P}_\pi v=\vert G\vert^{-1} \sum_{g\in G} \pi_g v$,
which is implemented by the uniform probability on $G$,
$u=\vert G\vert^{-1}\sum_{g\in G}g \in \CC G$.
Our formula agrees with this one since in this case, under mild assumptions, the elements $\rho^k$ satisfy $\Vert \rho^k - u\Vert_1 \to 0$,
where $\Vert \rho\Vert_1=\sum_{g\in G} \vert \rho (g)\vert$ is the $\ell_1$-norm of $\rho $, see e.g. \cite{diaconis-saloff-coste}.
\end{remark}
\endcomment

\begin{remark}[Constructing almost invariant vectors]\label{rem:algo}\normalfont
The convergence with controlled speed of iterated Markov operators to the projection onto the space of fixed vectors allows to produce almost invariant vectors with an 
arbitrarily small degree of almost invariance. In particular, it allows to design a non-deterministic algorithm which, given a vector, never stops if the vector has no component in the subspace of fixed vectors, while if it stops it produces an almost invariant unit vector with the desired degree of almost invariance (equivalently, an approximation with the desired order of error of a fixed unit vector).  

The estimates on the norm of Markov operators also allow to compute explicit mixing times, which would 
again be uniform for all vectors in $E_\pi$.

Note that this can be achieved not only for finitely generated groups, but also for topological groups that admit finite Kazhdan sets (see section \ref{sec:Kazhdansets}).

\end{remark}

\subsection{From projections to spectral gaps}\label{section: from projections to Kazhdan constants}

Finally, we show that a summable convergence of $\left( A_\pi^\mu\right)^k$ to the projection onto the subspace of fixed points implies the existence of a spectral gap.

\begin{theorem}\label{thm:casm}
Let $\mu $ be a compactly supported probability measure on $G$ and $\mathcal{F}$ be a complemented family of  isometric representations of $G$ on a family
$\mathcal{E}$ of Banach spaces.
Assume that there exists a number $\mathfrak{S}<\infty$ such that 
for every representation $\pi\in\mathcal{F}$ on $E\in \mathcal{E}$ the iterated Markov operators $\left( A_\pi^\mu\right)^k$ converge to a bounded operator $\mathcal{P}$ and 
\begin{enumerate}
\item $E_\pi\subseteq \ker \mathcal{P}$,
\item the convergence has speed summable to at most $\mathfrak{S}$, i.e. there exists a sequence of positive numbers $a_k$ such that the series $\sum_{k\in \N } a_k$ 
converges to a finite number $\le \mathfrak S$ and 
\begin{equation}
\left\| \left( A_\pi^\mu\right)^k - \mathcal{P} \right \| \le a_k.
\end{equation}
\end{enumerate}
Then the family  $\mathcal{F}$ has a spectral gap and $\left( \supp \mu, \frac{1}{ 1+ {\mathfrak{S}}} \right)$ is Kazhdan pair.
\end{theorem}
\begin{proof}

Denote $Q=\supp\mu$. 
Let $\pi\in \mathcal{F}$ be an isometric representation of $G$ on $E\in\mathcal{E}$ and
$v\in E_\pi$ be an arbitrary unit vector. Let $\sigma_v=\sup_{g\in Q} \| \pi_g v -v \|$. Then 
\begin{equation}\label{equation: bound on norm of Av-v}
\left\| A^\mu_\pi v -v \right\|\le \sigma_v.
\end{equation}
For every positive integer $k$ we can then write 
\begin{align*}
\left\| \left(A^\mu_\pi\right)^k v -v \right\|&\leq \left\| A^\mu_\pi v -v \right\| + \sum_{i=1}^{k-1}\left\| \left(A^\mu_\pi\right)^{i+1} v - \left(A^\mu_\pi\right)^{i} v\right\| \\
&\leq \left\| A^\mu_\pi v -v \right\| \left( 1+ \sum_{i=1}^{k-1}\left\| \left(A^\mu_\pi\right)^{i}|_{E_\pi}\right\|\right)\\
& \leq \sigma_v \left( 1+ \sum_{i=1}^{k-1}a_k\right).
\end{align*}
Then 
$$1 \leq \left\| \left(A^\mu_\pi\right)^k v -v \right\| +  \left\| \left(A^\mu_\pi\right)^k v \right\| \leq \sigma_v \left(1+ {\mathfrak{S}}\right) + a_k .$$
Since the above is true for every $k\in \N$ we obtain
$$\sigma_v \geq \dfrac{1}{1+ {\mathfrak{S}}}. 
$$
\end{proof}

The above estimates have several interesting consequences. The first is that as soon as the upper bounds 
$\alpha_k =  \left\| \left( A_\pi^\mu\right)^k - \mathcal{P}  \right \|$
compose a convergent series, they must be decreasing exponentially, $\mathcal{P}$ must be the projection onto 
$E^\pi$ along $E_\pi$, and we are in the presence of a spectral gap.

Another consequence of the previous results is that the Kazhdan constant and the  norm $\left\| A_{\pi}^{\mu}|_{E_{\pi}} \right\|$ are closely related. 
We will state this result for finite Kazhdan sets and uniform measures, since in this case the formulation is particularly concise. The general case
can be deduced in the same manner.

\begin{theorem}\label{thm:existmu}
Let $G$ be a locally compact group and let $\mathcal{F}$ be a family of isometric representations of $G$ on a family $\mathcal{E}$ of uniformly convex Banach spaces. 
Assume that $\mathcal{F}$ has a spectral gap and let  $(Q,\kappa )$ be a Kazhdan pair for $\mathcal{F}$, where $Q$ is finite. 
If $\mu$ is the uniform measure on $Qg \cup \lbrace g \rbrace$, for an arbitrary element $g\in G$ then for every representation $\pi\in \mathcal{F}$ we have
$$1-\kappa  \le  \left\| A_\pi^\mu |_{E_\pi} \right\| \le 1-  \dfrac{2}{\#Q+1}\delta_E (\kappa ).$$
\end{theorem}
\begin{proof}
The upper bound follows from Theorem \ref{theorem in intro: spectral gaps and projections}, in particular Corollary \ref{corollary: main theorem for finitely generated and uniform measure}.
To prove the lower bound note that for a unit vector $v\in E_\pi$ the inequality (\ref{equation: bound on norm of Av-v}) yields 
$$1-\sigma_v= \lVert v\rVert - \sigma_v \le \lVert A_\pi^\mu v\rVert.$$ Passing to the supremum over $v\in E_\pi$ of norm 1 on both sides we obtain the claim.
\end{proof}

\begin{remark}\label{rmk:existmu}\normalfont
In the particular case when $\mathcal{E}$ is a family of Hilbert spaces the lower bound  on the Kazhdan constant in terms of the norm of the Markov operator $\kappa \geq 1- \left\| A_\pi^\mu |_{E_\pi} \right\|$  can be improved to
$$
\kappa \geq \sqrt{2} \sqrt{1- \left\| A_\pi^\mu |_{E_\pi} \right\|}\, .
$$
This is obtained using the argument in \cite[p. 842]{shalom-kazhdan}.
\end{remark}

A problem formulated by Serre and de la Harpe-Valette \cite{burger,harpe-valette-ast} asks to compute explicit Kazhdan (sets and) constants, 
for unitary representations on Hilbert spaces. In the case of representations on uniformly convex Banach spaces, Theorem \ref{thm:existmu} implies the following.

\begin{corollary}\label{cor:nearone}
Let $G$ be a locally compact group and let $\mathcal{F}$ be a complemented family of isometric representations of $G$ on Banach spaces for which $Q$ is a finite Kazhdan set. 
Then for every $\varepsilon >0$ there exists an integer $m\in \N$ such that   $(\overline{Q}^m, 1-\varepsilon)$ is a Kazhdan pair for $\mathcal{F}$, where $\overline{Q} = Q \cup \lbrace e \rbrace$.
\end{corollary}

\proof Given $\mu$ the uniform probability measure on $\overline{Q}$ and $\pi$ an arbitrary representation in $\mathcal{F}$, 
by Corollary \ref{cor:amufg}  we have $\left\| \left( A^\mu_\pi \right)^k|_{E_\pi} \right\|= \left\|  A^{\mu^k}_\pi|_{E_\pi} \right\|\leq \lambda^k$ for some $\lambda \in (0,1)$, where $\mu^k = \mu \ast \dots \ast\mu$. Since the support of the probability measure $\mu^k$ is $X= \overline{Q}^k$, the argument in the proof of Theorem \ref{thm:casm}, which yields the inequality in Theorem \ref{thm:existmu}, implies that the Kazhdan constant for $X$ is at least $1- \lambda^k$.
\endproof

\comment
\begin{remark}\normalfont
 Arguments as in the proof of Corollary \ref{cor:nearone}, results in \cite{hadad} and  Remark \ref{rmk:existmu} imply the following.
\begin{enumerate}
\item Every group $EL_n(R)$ generated by all $n\times n$ matrices $e_{ij}(x)$ with diagonal entries $1$ and the $(i,j)$ entry $x\in R$ (where $R$ is an associative ring with unit generated by $\ell$ elements and with stable range $r\leq n-1$) has a Kazhdan pair $(X_n, \sqrt{2}-\varepsilon )$, with $\varepsilon >0$ arbitrary and $\# X_n \leq k=k(\varepsilon , \ell , r)$ 

\medskip 

\item  For every $\varepsilon >0 $, every group $SL_n(\Z [x_1, . . . , x_\ell ])$ has a generating set of cardinality at most $k$ and Kazhdan constant at least $\sqrt{2}-\varepsilon$, where $k = k(\ell, \varepsilon )$.

\medskip

\item For every ring of integers $\mathcal O$ in a global field, and every $\varepsilon >0 $ there exist $k = k({\mathcal O} ,\varepsilon )$ such that for every $n \geq 3$, $SL_n({\mathcal O})$ has a generating set of cardinality at most $k$ and Kazhdan constant at least $\sqrt{2}-\varepsilon$.
\end{enumerate}

\end{remark}
\endcomment

We additionally remark that in an appropriate setting one can even achieve similar results for families of groups, provided the constants appearing in the above statements
can be arranged to be uniform.

\subsection{Finite Kazhdan sets}\label{sec:Kazhdansets}

For many applications the existence of finite Kazhdan sets is a considerable asset, as the averages become finite, the random walks discrete, and an algorithmical approach and the use of computer become possible (see for instance Theorem \ref{theorem: definition of the projection onto invariant vectors}, Remark \ref{rem:algo} and Section \ref{sec:ergodic}). As it turns out, the existence of such finite sets is granted in many cases. Shalom formulated a property that he called the {\emph{strong property}} $(T)$ (st.pr. $(T)$), requiring the existence of a finite Kazhdan set, and proved that many property $(T)$ groups satisfy it. This theme meets another more recent one, which is the existence of a spectral gap of Hecke operators \cite{bourgain-gamburd,bourgain-gamburd-jems}.

In \cite{shalom-kazhdan}, Y Shalom proved st.pr. $(T)$ for groups of $k$--points of a simply connected, semisimple, almost $k$-simple group 
of rank at least $2$ (where $k$ is a locally compact non-discrete field), with explicit descriptions of finite Kazhdan sets and their corresponding constants. 
Theorem C in \cite{shalom-kazhdan} implies that in a semisimple Lie group with finite center, every finite symmetric set not contained in a closed amenable subgroup is a Kazhdan set.  In a second paper \cite{shalom-tams}, Shalom proved that any 
connected Lie group with property $(T)$ that does not have $\R/\Z$ as a quotient has st. pr (T).

\medskip

In the case of a compact group $G$, st. pr $(T)$ has a very interesting equivalent. In that case, for the regular representation of $G$ on $L_2(G)$, 
when the measure $\mu$ is supported on a finite symmetric set $\set{g_1^{\pm 1},\dots, g_m^{\pm 1}}\subset G$, the averaging operator $2mA_\pi^\mu$ is also known 
as the {\emph{Hecke operator}} and is sometimes also denoted $z_{g_1,\dots, g_m}$.
This operator is said to have a {\emph{spectral gap}} if its norm on the space $L_2^0 (G)$ of functions orthogonal to the constants is at most $2m- \zeta$.

The following double implication, which essentially amounts to an equivalence between spectral gap and $\set{g_1^{\pm 1},\dots, g_m^{\pm 1}}$ being a Kazhdan set, then holds:

\begin{enumerate}
\item if the Hecke operator $z_{g_1, \dots ,g_m}$ satisfies $\left\lVert z_{g_1,\dots,g_m} \right\rVert\le 2m- \zeta$ on $ L_2^0(G)$, where $\zeta>0$, then $\left(\set{g_1, \dots ,g_m}, \sqrt{\zeta /m}  \right)$  is a Kazhdan pair;

\medskip

\item given a Kazhdan pair $(Q, \epsilon )$, and an arbitrary $g\notin Q$, the Hecke operator $z_{\{ g\} \cup Qg}$ has spectral gap with $\zeta\ge 2m - 2 + \sqrt{4-\kappa^2 }$. 
\end{enumerate} 

Results of Bourgain and Gamburd \cite{bourgain-gamburd, bourgain-gamburd-jems} then give, \textit{via} the equivalence described above, many explicit finite Kazhdan sets for the groups $\operatorname{SU}(d)$, with Kazhdan constants explicitly computable from constants appearing in a certain noncommutative Diophantine property satisfied by the given set.


\section{Kazhdan projections in Banach algebras}

\subsection{Group Banach algebras and projections}

Let $G$ be a locally compact group and denote by $C_c(G)$ the group algebra of compactly supported continuous functions on $G$ with convolution.
Let $\mathcal{F}$ be a family of representations $\pi:G\to \mathcal{B}(E)$, by bounded operators on Banach spaces $E$ in a given family $\mathcal{E}$. 
For the purposes of this section we also assume that $\mathcal{F}$ contains the trivial representation on at least one $E\in \mathcal{E}$.

Assuming that
for each $f\in C_c(G)$, $\sup\setd{\Vert \pi(f)\Vert_{{\mathcal{B}}(E)}}{\pi \in \mathcal{F}}<\infty$,
we equip the algebra $C_c(G)$ with the norm
$\Vert f\Vert_{\mathcal{F}}= \sup_{\pi\in\mathcal{F}} \Vert \pi(f)\Vert_{{\mathcal{B}}(E)}$.
\begin{definition}\label{def:cf}
The algebra $C_{\mathcal{F}}(G)$ is the completion of $C_c(G)$ in the norm $\Vert \cdot \Vert_{\mathcal{F}}$.
\end{definition}


If $\opi\in \mathcal{F}$ whenever $\pi\in \mathcal{F}$  and the class $\mathcal{E}$ is stable under complex conjugation then $C_{\mathcal{F}}(G)$
admits a natural involution.

The classical example of algebra of type $C_{\mathcal{F}}(G)$ is the maximal group $C^*$-algebra $C^*_{\max}(G)$, 
corresponding to $\mathcal{F}$ being the family of all unitary representations of $G$. Other examples include the following algebras.
\begin{itemize}
\item The \emph{$L_p$-maximal group algebra}, where $p\in (1, \infty )$, denoted by $C_{\max}^p(G)$. This algebra corresponds to the class $\mathcal{F}$ of all 
isometric representations of $G$ on $L_p$-spaces. 
\item \emph{Uniformly bounded group algebras}, corresponding to the choice of $\mathcal F$ as a class $[\mathcal{E};k]$ composed of all the uniformly bounded representations of $G$ on $E\in \mathcal{E}$
satisfying $\Vert \pi \Vert \le k$, where $k\ge 1$ and $\mathcal{E}$ is a uniformly convex family of Banach spaces.

\item \emph{Small exponential growth group algebras}. The following classes of representations and algebras were introduced by Lafforgue in \cite{lafforgue-duke}. 
Let $\ell$ be a continuous length function on $G$, and 
let $\mathcal{E}$ be a family of Banach spaces closed under duality and complex conjugation. A representation $\pi : G \to {\mathcal{B}}(E)$ on $E\in\mathcal{E}$ is said to have $(s,c)$--{\emph{small exponential growth,}} for some $s,c>0$, if $\Vert \pi_g\Vert\le c e^{s\ell(g)}$ for every $g\in G$.
We denote the class of all such representations by $\mathcal{L}(\ell,s,c)$, and we call the algebra $C_{\mathcal{L}(\ell,s,c)}(G)$ 
a {\emph{small exponential growth algebra}}. 

\end{itemize}

The  algebra $C_{\max}^p(G)$ is an immediate natural generalization of the maximal group $C^*$--algebra, see \cite{phillips,gardella}.
Such algebras are relevant for an $L^p$--approach to the Novikov and the Baum-Connes conjectures \cite{kasparov, chong}.

The algebra $C_{[\mathcal{E};k]}$ and the corresponding notions of property $(T)$ for uniformly bounded representations are related to a conjecture of Y. Shalom that every hyperbolic group has a proper affine action on a Hilbert space with linear part uniformly bounded, see \cite[Problem 14]{oberwolfach} and \cite{nowak} for related results. This conjecture has attracted a lot of interest lately.

\begin{definition}
A {\emph{Kazhdan projection in $C_{\mathcal{F}}(G)$}} is a central idempotent $p \in C_{\mathcal{F}} (G)$ such that $\pi(p)=\mathcal{P}_\pi$ for every representation $\pi\in \mathcal{F}$.
\end{definition}

Kazhdan projections are important already in the setting of unitary representations \cite{connes,higson-roe}. Their existence in certain algebras $C_{\mathcal{L}(\ell,s,c)}(G)$ is also particularly significant, as they are used by V. Lafforgue to define strong Banach property $(T)$. 
The latter property is relevant to the Baum-Connes conjecture, see 
\cite{lafforgue-quanta,puschnigg}.

\begin{definition}[\cite{lafforgue-duke}]\label{def:laff}
The group $G$ has the {\emph{strong Banach property $(T)$ for $\mathcal{E}$}}, denoted $(T^{Ban}_{\mathcal{E}})$, 
if for every continuous length function $\ell$ on $G$ there exists $s>0$ such that for every $c>0$ the algebra
$C_{\mathcal{L}(\ell,s,c)}(G)$ contains a Kazhdan projection.
\end{definition}

In the case of representations with small exponential growth we record  the following

\begin{theorem}\label{theorem:Lafforgue}
Let $G$ be a locally compact group, and $\mathcal{E}$ a class of Banach spaces  closed under duality and complex conjugation.
\begin{enumerate}
\item The following are equivalent:
\begin{enumerate} 
\item $G$ has the property $(T^{Ban}_{\mathcal{E}})$; \label{theorem: Laf item i-a}
\item for every continuous length function $\ell$, there exists $s>0$ such that for every $c>0$, $\mathcal{L}(\ell,s,c)$ is complemented, and there exists $\rho \in C_c(G)$ 
satisfying $\int_G \rho {\mathrm{d}}\eta = 1$, and $\lambda<1$, such that $\left\lVert A_\pi^\rho \vert_{E_\pi}\right\rVert<\lambda$ for every  $\pi \in \mathcal{L}(\ell,s,c)$;
\label{theorem: Laf item i-b}
\item for every continuous length function $\ell$, there exists $s>0$ such that for every $c>0$, $\mathcal{L}(\ell,s,c)$ is complemented, 
and there exists $\rho$ and $\lambda$ as above such that $\rho^n$ converge exponentially fast to some element $p\in C_{\mathcal{L}(\ell,s,c)}(G)$,
$$\left\| \rho^n - p \right\|_{\mathcal{L}(\ell,s,c)} \le \lambda^n \, .$$
\label{theorem: Laf item i-c}
\end{enumerate}
\item If $G$ has the property $(T^{Ban}_{\mathcal{E}})$ then for the corresponding $\ell$ and $s$ and an arbitrary $c>0$, the pair $\left( \supp \rho ,\, \frac{1-\lambda}{a_\rho} \right)$ is a Kazhdan pair for the family ${\mathcal{F}}=\mathcal{L}(\ell,s,c)$, where $a_\rho = \int \left| \rho (g)\right|\, {\mathrm{d}}\eta (g) $.
\label{theorem: Laf item ii}
\end{enumerate}
\end{theorem}

\begin{remark}\normalfont 
The main difference in comparison to Theorem \ref{theorem: main theorem in text spectral gaps and projections}, is that the Markov operators are defined by
signed measures. It is unclear if $\rho$ can always be chosen to be non-negative in this setting.
In all the cases in which Kazhdan projections have been constructed in the small exponential growth algebras, they are in fact 
limits of positive functions with compact support and of integral $1$ \cite{lafforgue-duke,liao,delasalle}.
\end{remark}

\begin{proof}  \eqref{theorem: Laf item i-a} $\Rightarrow$ \eqref{theorem: Laf item i-b}. 
The fact that $\mathcal{L}(\ell,s,c)$ is complemented follows from the fact that $\mathcal E$ is closed under duality.
There exist $p_n \in C_c(G)$ of integral $1$ converging to $p$ in $C_{\mathcal{L}(\ell,s,c)}(G)$. Therefore, one can choose $\rho = p_n$ for $n$ large enough.

\eqref{theorem: Laf item i-b} $\Rightarrow$ \eqref{theorem: Laf item i-c} is proved exactly as the similar implication in 
Theorem \ref{theorem: main theorem in text spectral gaps and projections}.

\eqref{theorem: Laf item i-c} $\Rightarrow$ \eqref{theorem: Laf item i-a}  is straightforward.

\eqref{theorem: Laf item ii} Using the equivalence \eqref{theorem: Laf item i-a} $\Leftrightarrow$ \eqref{theorem: Laf item i-c} 
and the argument in the proof of \ref{thm:casm}, one obtains the required conclusion. 
\end{proof}


In the case of isometric representations we have the following  characterization of the existence of Kazhdan projections. 

\begin{theorem}\label{theorem: Kazhdan projections in algebras defined via a class of reps}
Let $\mathcal{F}$ be a family of isometric representations of a locally compact group $G$ on a uniformly convex family $\mathcal{E}$. 
There exists a Kazhdan projection in $C_{\mathcal{F}}(G)$ if and only if the family $\mathcal{F}$ has a spectral gap.

Moreover, $p$ is always a limit of a sequence of positive probability measures.
\end{theorem}

\begin{proof} 
Let $Q$ be a Kazhdan set for $\mathcal F$, and let $\rho\in C_c(G)$ be a density function of a continuous admissible measure $\mu$ with respect to $Q$. It 
suffices to show that $\set{\rho^k}_{k\in \NN}$ is a Cauchy sequence with respect to the norm $\Vert \cdot \Vert_{\mathcal{F}}$.

For any representation $\pi\in \mathcal{F}$ and $m<n$ we again have that the image of $I_E-\left(A_\pi^\mu\right)^n$ is in $E_\pi$ and
\begin{align*}
\left\lVert \pi\left(\rho^m\right)-\pi\left(\rho^n\right)\right\rVert_{{\mathcal{B}}(E)} &= \left\lVert \left(A_\pi^\mu\right)^m \left(I_E-\left(A_\pi^\mu\right)^{n-m}\right)\right\rVert\\
&\le \left\lVert A_\pi^\mu\vert_{E_\pi}\right\rVert^m \left\lVert I_E-\left(A_\pi^\mu\right)^{n-m}\right\rVert\\
&\le 2 \left\lVert A_\pi^\mu\vert_{E_\pi}\right\rVert^m
\end{align*}
Since the last term satisfies a uniform estimate  $\left \lVert A_\pi^\mu\vert_{E_\pi}\right\rVert^m\le \lambda^m $ for some $\lambda\in (0,1)$ 
and every $\pi\in \mathcal{F}$, the sequence $\rho^k$ is indeed a Cauchy sequence
in the $\mathcal{F}$ norm,  as claimed, $\left\lVert \rho^m-\rho^n\right\rVert_{\mathcal{F}} \le \lambda^m$.
There exists a limit, denoted $p$, which is the required Kazhdan projection.
Indeed, a similar estimate as above shows that $\Vert \rho^{2m} - \rho^{m}\Vert_{\mathcal{F}} \le 2\Vert \rho^m\Vert_{\mathcal{F}}\to 0$,
hence $p$ is an idempotent. Finally, since $\pi(p)$ commutes with $\pi(g)=\pi_g$ for every $g\in G$ and every $\pi$ in $\mathcal{F}$,
we see that $p$ is central.

The converse follows from Theorem \ref{theorem:Lafforgue}, \eqref{theorem: Laf item ii}. It was also proved in \cite{lafforgue-duke}.
\end{proof}

As an application, we have the following generalization of the existence of Kazhdan projections to the class of $L_p$-maximal group algebras.
\begin{corollary}\label{corollary: Kazhdan projection and uniform (TE)}
If $\mathcal{F}$ is the family of all isometric representation of $G$ on $E$ then 
then there exists a Kazhdan projection in $C_{\mathcal{F}}(G)$ if and only if $G$ has uniform property $(TE)$.

In particular, if $G$ has Kazhdan's property $(T)$ then for every $1<p<\infty$ there exists a Kazhdan projection
in the $L_p$-maximal group algebra $C_{\max}^p( G)$.
\end{corollary}

All the statements in Theorems \ref{theorem:Lafforgue} and \ref{theorem: Kazhdan projections in algebras defined via a class of reps} are also true, with the appropriate changes, in the more general case of uniformly bounded representations.

In the case of finitely generated groups  it was shown in \cite{delaat-delasalle} that property $(TE)$
is equivalent to the fact that for every isometric representation $\pi$ of $G$ on $E$ the projection onto $E^\pi$ is in the closure 
of $\setd{\pi(f)}{\supp f<\infty}\subseteq \mathcal{B}(E)$.

\subsection{Relations between versions of property $(T)$ for Banach spaces}\label{section: relations between different notions of T}
The above results have another important consequence: for uniformly convex Banach spaces they
put on the common ground the reinforced Banach property $(T^{Ban}_{\mathcal{E}})$, introduced by V.~Lafforgue in \cite{lafforgue-duke}
and the properties $(TE)$ and $FE$, introduced in \cite{fisher-margulis,bfgm}. We refer to \cite{nowak-survey} for a survey of these properties.
The question of 
such a relation was considered by several experts. In particular, it appeared as Question 1.10 in \cite{chatterji-drutu-haglund}.

To make a direct comparison consider a uniformly convex family $\mathcal{E}$ of Banach spaces, closed under duality and complex conjugation, 
and the class $\mathcal{F}$ of all isometric 
representations of $G$ on all $E\in \mathcal{E}$.

Recall that $G$ has property $FE$ for a Banach space $E$ if every continuous affine isometric action of $G$ on $E$ has a fixed point \cite{fisher-margulis,bfgm}.
Equivalently, $H^1(G,\pi)=0$ for every isometric representation $\pi$ of $G$ on $E$. We will say that $G$ has property $F\mathcal{E}$ for a family of Banach 
spaces $\mathcal{E}$ if $G$ has property $FE$ for every $E\in \mathcal{E}$.

It follows from the results of this paper that we have the  implications:\\

\begin{tikzpicture}
[
    implies/.style={double,double equal sign distance,-implies},
    dot/.style={shape=circle,fill=black,minimum size=2pt,
                inner sep=0pt,outer sep=2pt},
]

\node at (0,0.5) { }; 
\node (LAF) at (0,3) {Lafforgue's $(T^{Ban}_{\mathcal{E}})$ };
\node (KPE) at (0,1) {Kazhdan projection in $C_{\mathcal{F}}(G)$};
\node (UTE) at (8,1) {uniform $(T\mathcal{E})$};
\node (FE) at (5,3) {F$\mathcal{E}$};
\node (TE) at (8,3) {$(T\mathcal{E})$};

\node at (4.7,1.3) {Theorem \ref{theorem: Kazhdan projections in algebras defined via a class of reps}};
\node at (3.1,3.3) {\cite{lafforgue-jta}};
\node at (6.3,3.3) {\cite{bfgm}};

\node at (0.8,2) {(by def.)};
\node at (8.8,2) {(by def.)};

\draw  [implies,double equal sign distance] (LAF) -- (KPE);
\draw  [implies-implies,double equal sign distance] (KPE) -- (UTE) ;
\draw  [implies,double equal sign distance] (UTE) -- (TE);
\draw  [implies,double equal sign distance] (LAF) -- (FE);
\draw  [implies,double equal sign distance] (FE) -- (TE);
\end{tikzpicture}

As mentioned earlier,  uniformity of property $(T\mathcal{E})$ is automatic if the class $\mathcal{E}$ is closed under taking infinite direct sums and then the vertical arrow on the right is an equivalence.
We also remark that Lafforgue's proof of the first
implication $(T^{Ban}_{\mathcal{E}}) \Longrightarrow F\mathcal{E}$ does in fact use linearly growing representations in an essential way \cite{lafforgue-jta}.

We also observe that the existence of a Kazhdan projection in $C_{\mathcal{F}}(G)$ does not in general imply $F\mathcal{E}$. The reason is  
the existence of hyperbolic groups with property $(T)$. Such groups all have $(TL_p)$ for every $1<p<\infty$, but for $p>2$ sufficiently large every hyperbolic group admits an
unbounded, or even a proper
affine isometric action on some $L_p$-space \cite{bourdon-pajot,yu,nica}.

\begin{remark}[Uniform property $FE$]\normalfont
We take this opportunity to remark that it is possible also to define uniform property $FE$, as we very briefly describe here. 
Consider a discrete group $G$. Property $FE$ is the same as vanishing of $H^1(G,\pi)$ for every isometric representation $\pi$ of $G$ on $E$. 
That is, the codifferential $\delta_\pi:E\to \ker d_\pi$ is onto,
where $\ker d_\pi$ is the space of 1-cocycles equipped with a norm induced by restricting the cocycle to the generating set  $S$ and viewing it as a function in 
$\ell_2(S)\otimes E=\ell_2(S;E)$. This on the other hand is equivalent to the adjoint map $\delta_\pi^*:(\ker d_\pi)^*\to E^*$ being bounded below, i.e.
$\left\lVert \delta^*\varphi \right\rVert\ge C_\pi \left\lVert \varphi\right\rVert$ for every $\varphi \in (\ker d_\pi)^*$ and a uniform $C_\pi>0$.
This fact is used extensively in \cite{bader-nowak,nowak} and we refer the reader to those articles for details.
We can now make the folllowing definition: \emph{$G$ has uniform property $FE$ if $G$ has property $FE$ and, additionally, there exists $C>0$ such that for each isometric representation $\pi$  of $G$ on $E$ we have  $C_\pi>C>0$}.
We note however that such a  uniform choice of $C_\pi$ appearing in the second condition is equivalent to having a uniform spectral gap for all isometric representations on $E$. Thus
uniform property $FE$ is simply property $FE$ with the additional condition that property  $(TE)$ (implied by $FE$) is satisfied uniformly.
\end{remark}

\subsection{Expanders and property $(\tau)$ for Banach spaces}\label{section: tau and expanders for Banach}

Existence of Kazhdan projections is particularly important in connection to the topic of expander graphs. Indeed, it is \textit{via} such projections that the main existing counter-example to the coarse Baum-Connes conjecture was constructed, using expanders \cite{higson-lafforgue-skandalis}. Our results allow to connect a  Banach space generalization of the notion of expanders to the existence of a Kazhdan projection.

The most natural way of defining expanders in the Banach space context is \textit{via} Poincar\'e inequalities. 

\begin{definition}\label{def:e-expanders}
Given a uniformly convex Banach space $E$, a sequence $\set{\Gamma_i}_{i\in\NN}$ of graphs is a sequence of $E$-{\emph{expanders}} if it satisfies a Poincar\'{e} inequality for $E$-valued functions uniformly; that is, there
exists a constant $\kappa>0$ such that  the {\emph{Poincar\'{e} inequality}}
$$\sum_{v\in \Gamma_i} \Vert f(v)-Mf\Vert_E^2\le \kappa \sum_{v\sim u} \Vert f(v)-f(u)\Vert^2_E$$
holds for every $f:\Gamma_i\to E$ for every $i\in\NN$.
\end{definition}

As in the classical case, a way of constructing expanders is by means of a version of property $(\tau)$ of A. Lubotzky \cite{lubotzky-book,lubotzky-notices}. Thus, let $G$ be a finitely generated group and $Q$ a finite subset of $G$. Assume that $G$
is residually finite, and consider a sequence ${\mathcal{N}}=\set{N_i}_{i\in \NN}$ of finite index normal subgroups, satisfying $\bigcap_{i\in\NN} N_i=\set{e}$.
Let $q_i:G\to G/N_i$ be the quotient map for every $i\in \NN$. For $v,u\in G/N_i$ we write $v\sim u$ if $v$ and $u$ are joined by 
an edge in the Cayley graphs $\operatorname{Cay}(G/N_i,q_i(Q))$.

Given a uniformly convex Banach space $E$, we denote by $\mathcal{N}(E)$ the family of representations $\pi^{\iip}$ of $G$ on the spaces
$\ell_2(G/N_i, E)$ given by $\pi^{\iip}_g f(v)=f(g^{-1}v)$.
The projection $E\to E^{\pi^{\iip}}$ is simply given by the average 
$M_i f=[G:N_i]^{-1}\sum_{v\in G/N_i} f(v)$.

\begin{definition}
Let $E$ be a uniformly convex Banach space. A residually finite finitely generated group $G$  has \emph{property $(\tau_E)$ with respect to 
$\mathcal{N}$} if the family
$\mathcal{N}(E)$ has a spectral gap.
\end{definition}

The following statement is an application of Theorem \ref{theorem in intro: spectral gaps and projections}
and shows that property $(\tau_E)$ gives the right kind of generalization of property $(\tau)$ to the setting of uniformly
convex spaces, and it yields a Kazhdan projection. 

\setcounter{aux1}{\value{section}}
\setcounter{aux2}{\value{theorem}}
\setcounter{section}{1}
\setcounter{theorem}{\value{expthm}}
\begin{theorem}\label{theorem: expanders theorem in the text}
Let $E$ be a uniformly convex Banach space, $G$ be a finitely generated residually finite group 
and let ${\mathcal{N}}=\set{N_i}$ a collection of finite index subgroups with trivial intersection.
The following conditions are equivalent:
\begin{enumerate}
\item $G$ has property $(\tau_E)$ with respect to ${\mathcal{N}}=\set{N_i}$ and a symmetric Kazhdan set $Q$; \label{enumi: tau_E in text}
\item the Cayley graphs $\operatorname{Cay}(G/N_i,Q)$ form a sequence of $E$-expanders;\label{enumi: E-expanders in text}
\item there exists a Kazhdan projection $p\in C_{\mathcal{N}(E)}(G)$.\label{enumi: kazhdan projection in text}
\end{enumerate}
\end{theorem}
\setcounter{section}{\value{aux1}}
\setcounter{theorem}{\value{aux2}}

Note that in the case when $E$ is a Hilbert space the algebra $C_{\mathcal{N}(E)}( G)$ is a $C^*$-algebra. In that case the existence of the Kazhdan projection appearing 
in the above theorem was hinted at in \cite{higson,higson-lafforgue-skandalis}.

\begin{remark}\normalfont
We also point out that in the setting of uniformly convex Banach spaces 
one more characterization of property $(\tau_E)$ is true, namely that $G$ has $(\tau_E)$ with respect to $\set{N_i}$ 
if and only if the cohomology
$H^1(G,\pi)=0$, where $\pi=\oplus \pi^{\iip}$. The proof given for Hilbert spaces in \cite{lubotzky-zuk} can be copied verbatim to the above setting as it uses 
only the Open Mapping Theorem and basic norm estimates.
\end{remark}


\section{Applications to Ergodic Theory}\label{sec:ergodic}

In this section, we study properties of a measure preserving ergodic action of a locally compact group $G$ on a probability space $(X, \nu)$. One of the first theorems in this setting, due to Oseledec \cite{oseledets}, states that the powers $\rho^k$ of an arbitrary density function of a probability measure $\mu$ on $G$ form an \emph{ergodic sequence}:
$$\pi(\rho^k)f \longrightarrow \int_X f \, d\nu(x),$$
both $\nu$-almost everywhere and in $L_p(X,\nu)$. See e.g. \cite{nevo} for an overview.

A special case of Theorem \ref{theorem in intro: spectral gaps and projections} is a quantitative ergodic estimate for actions with a spectral gap. Consider a collection $\mathcal E$ of uniformly convex  Banach spaces, and a number $p$ in $(1 , \infty )$. Theorem \ref{theorem in intro: spectral gaps and projections} implies that if  $(Q, \kappa )$ is a Kazhdan pair for the family $\mathcal{F}$ of isometric representations of $G$ on $L_p(X_,\nu;E)$, with $E\in {\mathcal E}$, induced by the measure preserving action of $G$, then for every $Q$--admissible probability measure $\mu$,
\begin{equation}\label{equation: convergence mean ergodic thm for a class F}
\left\Vert A_\pi^{\mu^k} f - \int_X f \, d\nu\right\Vert_p \le \lambda^k \Vert f\Vert_p,
\end{equation} where $\lambda<1$ depends only on $p$, the normalizing factor of $\mu$, the modulus of convexity of $\mathcal E$, and $\kappa$.

By a standard Borel--Cantelli argument, it follows that for almost every $x\in X$ there exists $k_x$ such that for $k\geq k_x$
$$
\left| A_\pi^{\mu^k} f(x) - \int_X f\right|\leq \lambda^{ k} \| f\|_p\, , 
$$ where $\lambda \in (0,1)$ is a constant slightly larger than the one in \eqref{equation: convergence mean ergodic thm for a class F}.

In the particular case of ${\mathcal E} =\{ \R \}$ and of a group $G$ with Kazhdan's property $(T)$, by Remark \ref{rmks:kazhdanlp}, \eqref{lp1}, the same measure $\mu$ on $G$ satisfies \eqref{equation: convergence mean ergodic thm for a class F} for all $p\in (1, \infty )$, with $\lambda \in (0,1)$ depending continuously on $p$ and converging to $1$ at $\infty$.

\medskip  

Such estimates can be used for instance for shrinking target problems of orbits of subgroups that have property $(T)$. 
The shrinking target problems are yet another way of understanding measure preserving ergodic actions of groups on (metric) measure spaces. They are particularly significant in the case of actions of subgroups $H$ of Lie groups $G$ on quotients $G/\Gamma$, where $\Gamma$ is a lattice in $G$. 
For semisimple Lie groups $G$ and finite volume  quotients $G/\Gamma$, the shrinking target problems can be classified following the position of the ``target'', which can be in the boundary at infinity or inside  $G/\Gamma$; or following the type of subgroup $H$ acting on $G/\Gamma $. Most existing results study actions of amenable subgroups $H$ of $G$ (most often, cyclic or one dimensional). 
The earliest results have been proved for targets at infinity and geodesic flows (i.e. actions of one dimensional subgroups $H$ composed of semisimple elements). The problem of finding the generic behavior of geodesic rays with respect to a shrinking target situated in a cusp was completely settled in \cite{sullivan:log ,  kleinbock-margulis:log}. The argument in \cite{kleinbock-margulis:log} uses theorems of Howe-Moore type and fast decay of
correlation coefficients, and the fact that the
characteristic function of a neighborhood of a cusp may be
replaced by a smooth function, without significant loss of information. 

When the ``target'' is not at infinity, but inside $G/\Gamma$, the problem becomes that
of finding the generic behavior of orbits of $H$ in terms of distance to
a fixed point; for instance, of finding the generic speed at which an orbit of $H$ approaches that point. In this case, the methods based on Howe-Moore theorems fail. Still, the generic behavior of geodesic rays with respect to a
point in locally symmetric spaces of rank one has been found in \cite{sullivan:log,maucourant}, using methods specific to rank one. The higher rank case remains
open. The question in full generality, of measuring the rate at which a typical orbit approaches a point in $G/\Gamma $, has been asked by Kleinbock 
and Margulis in \cite{kleinbock-margulis:log}.

Here we show that, as Theorem \ref{theorem in intro: spectral gaps and projections} provides a good way to average in a group $H$ with 
property $(T)$ (average that, in many ways, plays the part of the average on F\o lner sets for amenable groups), it also allows to answer 
shrinking target problems for orbits of subgroups $H$ of $G$ that have property $(T)$, in terms of the above mentioned average and, due to Remark \ref{rmks:kazhdanlp}, \eqref{lp1}, provide an estimate of the error term that is, in a way, the best possible.


Let $(Y, \nu 
)$ be a probability space. For every integrable function $f$ on $Y$ we denote by $Mf$ its mean, that is $Mf=\int_Y f \, d\nu$.  
Let $\Phi = \left\{ f_n : Y \to \R_+ \; ;\; n\in \N \right\}$ be a sequence of non-negative integrable functions on $Y$. 
For $N\in \N$ consider the partial sums of series
$$
\Sigma^N_{\Phi} (y) =\sum_{i=1}^N f_i (y) \ \ \ \ \ \ \text{ and } \ \ \ \ \ \ E^N_{\Phi} = \sum_{i=1}^N Mf_i.
$$    

\begin{lemma}[\cite{kochen-stone}, Chapter 1, Lemma 10 in \cite{sprindzuk}]\label{lemma:bc}
Let $Y$ be as above and let $p$ be a positive real number larger than $1$.
\begin{enumerate}
\item For $\mu$--almost every $y\in Y$ we have $\liminf_{N\to \infty } \frac{\Sigma^N_{\Phi} (y)}{E^N_{\Phi}} < \infty$.
\item Assume that $Mf_n\le 1$ for every $n\in \N$ and that there exists $C>0$ such that for every $N > M \ge 1$,
\begin{equation}\label{eq:bc-indep}
\int_Y \left|\sum_{i=M}^N f_i (y) - \sum_{i=M}^N Mf_i \right|^p \, d\mu \le C \sum_{i=M}^N Mf_i.
\end{equation}
Then for every $\varepsilon >0$ one has that for $\mu$--almost every $y\in Y$
$$\Sigma^N_{\Phi} (y) = E^N_{\Phi} + O\left( \left( E^N_{\Phi}\right)^{1/p} \log^{1+1/p + \varepsilon } \left( E^N_{\Phi}\right) \right).$$
\end{enumerate}
\end{lemma}

\proof The proof follows verbatim the one of Lemma 10 in \cite{sprindzuk}, except that on top of page 48 one has to apply H\"older's inequality instead of the Cauchy-Schwartz inequality. \endproof


Let $G$ be a locally compact group, and $\Gamma $ a lattice in it. We consider $G/\Gamma$ endowed with the probability measure $\nu$ induced by the Haar measure on $G$, properly renormalized. 

We now consider another locally compact group $\Lambda$ with property $(T)$, $S$ a Kazhdan set of $\Lambda$, and $\mu$ a probability measure on $\Lambda$ admissible with respect to $S$. Assume that $\Lambda$ acts on $G/\Gamma$ by transformations preserving $\nu$, and that the action is ergodic.

\begin{examples}\label{ex-ergodic}\normalfont
\begin{enumerate}
\item When $G$ is a semisimple group and $\Gamma$ is an irreducible lattice in $G$, every infinite subgroup $\Lambda$ of $G$ acts ergodically on $G/\Gamma$ \cite{zimmer-book}.
\item When $G= \R^n$, $\Gamma = \Z^n$, with $n\geq 2$, every subgroup $\Lambda$ of  $SL_n(\Z )$ containing a matrix whose eigenvalues are not roots of unity acts ergodically on $G/\Gamma = {\mathbb{T}}^n$ \cite[Ex. 4.2.11]{katok-book}
\end{enumerate}
\end{examples}

Denote by $X_n$ the random variable representing the $n$-th step of the random walk on $\Lambda$ determined by the measure $\mu$. For every $x\in G/\Gamma $ we write $X_n (x)$ to signify the element in $G/\Gamma $ obtained by applying the group element $X_n$ in $\Lambda$ to $x$.

Let $p$ be a positive number larger than $1$, and let $\pi$ be the standard representation of the group $\Lambda$ by linear isometries on the Banach space $E=L^p \left(  G/\Gamma \right)$. We consider the action of $\Lambda$ on $E$ to the right, 
that is $g\cdot f = f \circ g$. In particular, for $f={\mathbb{1}}_\Omega$ the characteristic function of a measurable set $\Omega$, $g\cdot {\mathbb{1}}_\Omega = {\mathbb{1}}_{g^{-1} \Omega}\, $.

Because the action of $\Lambda $ is ergodic, the space $E^\pi$ is composed of constant functions, while $E_\pi$ is composed of functions of the form 
$f-Mf$, for every $f\in L^p \left(  G/\Gamma \right)$ (due to the fact that the dual of $L^p \left(  G/\Gamma \right)$ can be identified with $L^q \left(  G/\Gamma \right)$, where $\frac{1}{p}+ \frac{1}{q}=1$). As $\Lambda $ has property $(T)$, it follows, by Theorem \ref{theorem: main theorem in text spectral gaps and projections} and \cite[Theorem A]{bfgm}, that there exists $\lambda \in (0,1)$ depending on the Kazhdan constant of $S$ and on $p$, such that $\left\| A^{\mu}_\pi |_{E_\pi } \right\|\le \lambda\, $.

Let $h$ be a non-negative integrable function on $G/\Gamma$. For an arbitrary probability measure $\alpha$ on $\Lambda$, if we consider the function $f = A^{\alpha }_\pi \left( h \right)$ then by Fubini's Theorem we can write 
\begin{align*}
M(f) & = \int_{G/\Gamma } f(x)\, {\mathrm{d}}\nu (x) = \int_{G/\Gamma }\left[\int_\Lambda h \circ g (x)\, {\mathrm{d}}\alpha (g)\right] \, {\mathrm{d}}\nu (x)\\
 & = \int_\Lambda \left[\int_{G/\Gamma } h (g x)\,  {\mathrm{d}}\nu (x)\right]\, {\mathrm{d}}\alpha (g) = M (h )\, .   
\end{align*}

In particular, the above is true for the function $f = A^{\mu^n}_\pi \left( {\mathbb{1}}_\Omega \right)$, where $\Omega$ is a measurable set in $G/\Gamma$. Note that for $x\in G/\Gamma $, we have that
$$f(x) = {\mathbb{P}} \left( X_n (x) \in \Omega \right) .$$

Consider now a sequence $\left( \Omega_n \right)_{n\in \N }$ of measurable sets in $G/\Gamma \, $, and the sequences of measurable functions $h_n = {\mathbb{1}}_{\Omega_n}$ and $f_n = A^{\mu^n}_\pi \left( {\mathbb{1}}_{\Omega_n} \right)$. We prove that the hypotheses of Lemma \ref{lemma:bc} are satisfied. The left hand side of the inequality \eqref{eq:bc-indep} can be bounded as follows 
\begin{align*}
\left\| \sum_{i=M}^N \left( A^{\mu^i}_\pi h_i - M{h_i} \right)  \right\|_p & \le \sum_{i=M}^N \lambda^i \left\| h_i - M{h_i} \right\|_p\\
&\le \left(\sum_{i=M}^N \lambda^{qi}\right)^{1/q}\,  \left( \sum_{i=M}^N \left\| h_i - M{h_i} \right\|_p^p \right)^{1/p}\, .
\end{align*}
The first inequality above uses the property that on the subspace $E_\pi$, composed of functions of the form $f-Mf$, the norm of $A^{\mu^i}_\pi$ 
is bounded by $\lambda^i$, the second uses H\"older's inequality.

Elementary calculations yield the following inequality (see for instance \cite[$\S 4$]{drutu-mackay}) 
$$
|a-b|^p \leq a^p +b^p + \left( 1+ p2^p \right) \max \left(a^{p-1}b\, ,\, ab^{p-1} \right)\, \mbox{ for every } a,b\in \R_+\, .
$$

This allows us to write 
$$
\int |h_i - M{h_i}|^p {\mathrm{d}}\nu (x) \leq \int h_i^p {\mathrm{d}}\nu (x) + M{h_i}^p + C_p M{h_i} \int h_i^{p-1}{\mathrm{d}}\nu (x) \leq (2+C_p) M{h_i}\, ,$$
 where $C_p=1+ p2^p$. In the last inequality above we used the facts that 
$h_i$ is the characteristic function of a set, hence $h_i^\alpha = h_i$ for every $\alpha \geq 1$, and that $M{h_i}\in (0,1)$, therefore $M{h_i}^\beta \leq M{h_i}$ for every $\beta \geq 1$. 

We may therefore write
$$
\left\| \sum_{i=M}^N \left( A^{\mu^i}_\pi h_i - M{h_i} \right)  \right\|_p^p\leq \frac{2+C_p}{(1-\lambda^q)^{p/q}}\sum_{i=M}^N M{h_i}\, .
$$

Lemma \ref{lemma:bc} then implies the following.

\begin{theorem}\label{thm:RW}
Let $G$ be a locally compact group, and $\Gamma $ a lattice in it. Let $\set{\Omega_n}$ be a sequence of measurable subsets in $G/\Gamma$. 

Assume that a locally compact group $\Lambda$ with property $(T)$ acts ergodically on $G/\Gamma$. Let $\mu$ be a probability measure on $\Lambda$ that is admissible with respect to a Kazhdan set, and let $X_n$ be the random variable representing the $n$-th step of the random walk defined by $\mu$.
\begin{enumerate}
\item If $\sum_n \nu (\Omega _n)$ is finite then for almost every $x\in G/\Gamma$
$$\sum_{n\in \N }\mathbb{P} \left( X_n (x) \in \Omega_n \right) < \infty .$$
\item If $\sum_n \nu (\Omega _n)$ is infinite then for every $\varepsilon >0$ and 
for almost every $x\in G/\Gamma$, 
$$\sum_{n\le N}{\mathbb{P}} \left( X_n (x) \in \Omega_n \right) = S_N + O(S_N^{\varepsilon})\, , $$ where
 $S_N = \sum_{n\le N} \nu (\Omega _n)$.
In particular, 
$${\mathbb{P}} \left( X_n(x) \in \Omega_n \right) > \nu (\Omega _n) - O\left( \nu ( \Omega _n )^{\varepsilon} \right)  \, 
$$ for infinitely many $n\in \N$.  
\end{enumerate}
\end{theorem}

A particular case of the above is when $\Omega_n$ is the ball around an arbitrary fixed point $x_0$ and of radius $r_n$. In that case, for $D$ the dimension of $G/\Gamma$ and $\dist$ a distance on $G/\Gamma$ induced by a left invariant Riemannian distance on $G$,   
 if $\sum_n r_n^D$ is infinite, then for every $\varepsilon >0$, almost every $x\in G/\Gamma$ satisfies 
$$\sum_{n\le N}{\mathbb{P}} \left( \dist(X_n (x), x_0) \le r_n \right) = R_N + O\left( R_N^{\varepsilon} \right),$$
where  $R_N = \sum_{n\le N} r_n^D$. While if $\sum_n r_n^D$ is finite then the above sum of probabilities is also finite.

Since $\Lambda$ has property $(T)$, it is compactly generated. Let $\dist_\Lambda$ be an arbitrary word metric on it, corresponding to the choice of a compact generating set. Let $\alpha$ be a probability measure on $\Lambda$ with compact support generating $\Lambda$ as a semigroup, and let $X_n$ be the corresponding random walk. 
Since $\Lambda$ is non-amenable, according to \cite{kesten,guivarch} there exists $a >0$ such that almost surely 
\begin{equation}\label{linearA}
\lim_{n\to \infty } \dfrac{\dist_\Lambda (e, X_n)}{n}=2a.
\end{equation}

This implies that 
$$\mu^n\left( \Delta_n \right) \geq 1- \varepsilon_n,$$ 
where $\Delta_n =\setd{g\in \Lambda}{\dist_\Lambda (e,g) >a n}$ and $\lim_{n\to \infty } \varepsilon_n =0$.

Consider the sequence of probability measures $\eta_n = \dfrac{1}{\mu^n \left( \Delta_n \right)} \mu^n|_{\Delta_n }$. The previous argument, with the sequence $\mu^n$ replaced by $\eta_n$, gives the same results as above, but with ${\mathbb{P}} \left( X_n (x) \in \Omega_n \right)$, replaced by ${\mathbb{P}} \left( X_n (x) \in \Omega_n , \dist_\Lambda (X_n, e)\geq an  \right)$, respectively with ${\mathbb{P}} \left( \dist(X_n (x), x_0) \le r_n \right)$ replaced by ${\mathbb{P}} ( \dist(X_n (x), x_0) \le r_n ,$\\ $ \dist_\Lambda (e , X_n )\geq an )$. 

\medskip

We now provide a sample of applications of Theorem \ref{thm:RW}, to emphasize the type of information this theorem provides.

As recalled in the beginning of the section, most of the previously known shrinking target results concerned geometric objects such as geodesics and, later on, unipotent orbits, also called horocycles in the literature. Our results allow to provide information about another significant set of geometric objects: the horospheres. 

Recall that in a CAT(0) space $X$, every geodesic ray $\rho :[0,\infty )\to X$ defines a {\emph Busemann function} $f_\rho (x) = \lim_{t\to \infty } \left[ \dist (x,\rho (t)) - t \right]$. Its open sublevel sets $f_\rho < a$, with $a\in \R$, coincide with the union of open balls $B(\rho (t), t + a )$ and are called {\emph{open horoballs determined by $\rho $}}; the topological boundary of an open horoball is a level hypersurface $f_\rho = a$, and it is called {\emph{horosphere}}. Two geodesic rays at finite Hausdorff distance determine the same collection of horoballs and horospheres, possibly corresponding to different parameters $a\in \R$  \cite{BridsonHaefliger}. The point at infinity $\rho (\infty )\in \partial_\infty X$ is called {\emph{the point at infinity of the horoball (horosphere)}}.  

In the particular case when $X$ is a symmetric space of non-compact type and of rank at least two, it is known that its group of isometries $G$ does not act transitively on $\partial_\infty X$: the quotient $\partial_\infty X/G$ can be identified with a spherical simplex of maximal dimension in the spherical building structure of $\partial_\infty X$. The image of $\rho (\infty )$ in $\partial_\infty X/G$ is called the {\emph{slope}} of $\rho$ (respectively of the point at infinity of the horoball/horosphere). When $\rho (\infty )$ is one of the vertices of the simplex $\partial_\infty X/G$, we say it is {\emph{maximal singular}}. For more information we refer the reader to \cite{BridsonHaefliger}.  

In rank one symmetric spaces the stabilizers of horospheres are nilpotent groups, in higher rank on the other hand they are more complicated, with semisimple subgroups and unipotent radicals. Moreover the horospheres have explicit descriptions in terms of flags from the flag boundary (see for instance \cite[$\S 3$]{drutu-quadr} for descriptions of horospheres and their stabilizers in some particular cases). 

As mentioned before, while for amenable groups, in all the von Neumann type theorems or Birkhoff type theorems the right approach is to consider limits, either in $L^2$--norm or pointwise, of averages on F\o lner sets (in particular on balls, when the groups have sub-exponential growth), in the opposite case of property (T) groups it seems more appropriate to consider limits of averages with respect to sequences of measures approximating the Kazhdan projection. Many stabilizers of horospheres in higher rank have property (T) (in particular those we consider below), therefore the averages provided by the sums in Theorem \ref{thm:RW} are an appropriate approach for them. 

With these observations, we proceed to examine the particular case when $G = SL(s, \R)$, with $s\geq 3$, and $\Gamma = SL(s, \Z )$. The symmetric space corresponding to $G$ is $\calp_s = SO(s) \backslash SL(s, \R )$, which can be identified to the space of positive definite quadratic forms of volume $1$ with the action of $G$ by isometries to the right. Consider the locally symmetric space $\calt_n = \calp_n / \Gamma$ and $\pi : \calp_n \to \calt_n$ the natural projection. In the applications of Theorem \ref{thm:RW}, there are two types of choices involved:

\begin{enumerate}

\item\label{chx1} the choice of the sequence of shrinking sets $(\Omega_n)$; 

\medskip

\item\label{chx2} the choice of a Kazhdan set $Q$ for the acting group $\Lambda$, and of a measure $\mu$, admissible with respect to $Q$, determining the random walk. 

\end{enumerate}  

As far as the choice \eqref{chx1} is concerned, we only consider two cases:

\textbf{[$\Omega $.A]} Consider as basepoint $x_0$ in $\calt_s$ the image by $\pi $ of the canonical quadratic form $\q_0 (x) = x_1^2 + \cdots + x_s^2$ on $\R^s$, and consider the standard metric defined on $\calp_s$, and the metric it induces on $\calt_s $ (see \cite{drutu-transf, drutu-quadr}).

Given $D$ the dimension of $\calp_s$ and $\calt_s$, consider $r_n = \frac{1}{n^{1/D}}$, and $B_n$ the ball in $\calp_s$ centered in $\q_0$ and of radius $r_n$. For $n$ large enough both $B_n$ and $\pi (B_n)$ have measure $\asymp \frac{1}{n}$. 

The fact that a quadratic form $\q $ satisfies $\pi (\q ) \in \pi( B_n )$ implies that there exists a basis $\{ v_1, \dots , v_n \}$ of $\Z^n$ that is orthogonal with respect to the bilinear form defined by $\q$, and such that $\q (v_i)= \lambda_i\in \left[ e^{-r_n}, e^{r_n} \right]$. For simplicity we replace  $\pi (B_n)$ with a set $\Omega_n \subset \calt_s$ of approximately the same volume, defined by the condition that a quadratic form $\q $ satisfies $\pi (\q ) \in \Omega_n$ if there exists $\gamma \in \Gamma $ such that for every vector $v\in \R^s$, $\q (v) \in \left[ e^{-r_n}, e^{r_n} \right] \|\gamma v \|^2$. 
In particular $\left| \q (v) - \|\gamma v \|^2 \right\|\leq (r_n + o(r_n)) \|\gamma v \|^2$. We call such a form an $r_n$--{\emph{ almost rational form.}} 

A consequence of almost rationality is for instance an estimate for the convergence to zero of the gaps between successive values in $\q (\Z^n)$, for an irrational positive definite quadratic form $\q $. It was conjectured by Davenport and Lewis \cite{DavLew} that for such a form given $s\in \q (\Z^n)$ and $n(s)$ the minimum of $\q (\Z^n) \cap (s, \infty )$, the gap $n(s)-s$ converges to $0$ as $s\to \infty$. The conjecture was settled by Bentkus and G\"otze for $s\geq 5$ \cite{bentkusgotze, gotze}. For an $r_n$--almost rational form, for every $\epsilon >0 $ the value $s_0$ such that $n(s) -s < \epsilon$ for $s\geq s_0$ must satisfy $s_0 \geq \frac{1-\epsilon }{2}n^{1/D}$.

\medskip

\textbf{[$\Omega $.B]} Consider one of the Busemann functions defined in \cite[$\S 3.2$]{drutu-quadr}, say $f_{r_1} (\q )= \sqrt{\frac{s}{s-1}} \ln \q (e_s)$. Its restriction to $\calt_s$ seen as a fundamental domain in $\calp_s$ is also a Busemann function, corresponding to a geodesic ray contained in $\calt_s$ \cite[$\S 3.6$]{drutu-quadr}. Let us denote the latter Busemann function $\bar{f}$, and let $(\Omega_n )$ be a sequence of thinner and thinner slices of level sets, defined by $-r_n - \epsilon_n \leq \bar{f} \leq -r_n$. A standard calculation of volume implies that in order to obtain a slice of volume $\frac{1}{n}$, one has to choose $\epsilon_n = \frac{1}{n^\delta }$ for some $\delta \in (0,1)$ and $r_n= \eta \ln n$, where $\eta $ depends on $s$, on the choice of $\delta $ and of the ray defining the Busemann function.

\medskip

In relation with the choice \eqref{chx2}, we begin by recalling that any symmetric set not contained in a closed amenable subgroup is a Kazhdan set for $\Lambda$ \cite[Theorem C]{shalom-kazhdan}. We shall mainly consider the following types of choices:

\medskip

\textbf{[$Q $.1]} The group $\Lambda$ is $G$, and $Q$ is a compact set generating $G$ (possibly a ball for some left-invariant metric on $G$).

\medskip

\textbf{[$Q $.2]} The group $\Lambda$ is the stabilizer of a horosphere in the symmetric space $X$ associated to $G$, with basepoint of maximal singular slope, and $Q$ is a compact set generating it. 

Note that as the action of $G$ on $X$ is to the right, as defined in \cite[$\S 3$]{drutu-quadr}, when $\Lambda$ is the stabilizer of a horosphere $f_\rho =0$, for an arbitrary quadratic form $\q$, $\q \Lambda$ is the horosphere $f_\rho =a$ containing $\q$, and $\q \Lambda g$, for $g\in G$, is the image by the isometric action of $g$ of the horosphere $f_\rho =a$; for every $n\in \N$, $\q Q^n g$ is a large compact subset of the latter.  

In what follows we describe how various choices produce different significant results:

Choices \textbf{[$Q $.1]}  and \textbf{[$\Omega $.A]}: Given a random walk on the group $G$, the probabilities that at step $n\in [1,N]$ the corresponding quadratic form is $\frac{1}{n^{1/D}}$--almost rational sum up to a value of $\ln N + O(\ln^\varepsilon N)$, for any $\varepsilon >0$. In particular infinitely many times such a probability is above $\frac{1}{n}$. 

\medskip

Choices \textbf{[$Q $.2]}  and \textbf{[$\Omega $.i], for $i\in \{ A,B\}$}: Let $x$ be an arbitrary point in $\calt_s$. Almost every horosphere $\mathcal H$ in the symmetric space $\calp_s$ ($\mathcal H$ with point at infinity of a fixed given slope that is maximal singular), and almost every point $h\in \mathcal H$, have the property that the $\pi $ projection of the annulus ${\mathcal{H}} \cap \left[ B(h, n)\setminus B(h, an) \right]$ inside  $\calt_s$ intersects the shrinking ball $\pi \left( B(x, n^{1/D})\right)$ for infinitely many $n$ (where $a>0$ is the constant provided by \eqref{linearA} for the group $\Lambda$). The same holds if one replaces the condition above with the one that the $\pi $ projection of ${\mathcal{H}} \cap \left[ B(h, n)\setminus B(h, an) \right]$ rises infinitely many times into the cusp of $\calt_s$ at a height between $\eta \ln n$ and $\eta \ln n +\frac{1}{n^\delta}$, where the height in the cusp is measured by a certain fixed Busemann function (the constant $\eta$ depending on this choice).

A general version of the result, for horospheres in general symmetric spaces, can likewise be obtained. 

A number theoretical interpretation of the above (for appropriate choices of $x\in \calt_s$, the maximal singular slope and the Busemann function on $\calt_s$) is that for every $1\leq k <s$, for almost every $k$--dimensional subspace $W$ in $\R^s$, there exist a sequence $(\q_n)$ of quadratic forms farther and farther away from $\q_0$, with restrictions to $W$ of the same volume, and $\frac{1}{n^{1/D}}$--almost rational (respectively, such that for some primitive vector $v\in \Z^n$  , $\q (v)\in \left[ \frac{1}{n^\xi }\left(1-\frac{\kappa }{n^\delta } \right) ,\frac{1}{n^\xi }\right]$, where $\delta \in (0,1)$, $\kappa$ depends on $s$, $\xi >0$ depends on $\delta$ and $s$).


\section{Ghost projections for warped cones}\label{section: ghost projections}

In this section we construct new examples of non-compact ghost projections.
Such projections associated to expanders are the source of all known 
bounded geometry counterexamples to the coarse Baum-Connes conjecture \cite{higson,higson-lafforgue-skandalis,willett-yu}.
In fact, no  examples of ghost projections other than the one described in  \cite{higson}, \cite[Examples 5.3]{willett-yu} were known until now.

Warped cones are unbounded spaces constructed by successive approximations of an action of a group on some space. 
They share certain characteristics with box spaces (i.e., sequences of finite quotients of a group) and in this sense our construction is analogous to the 
construction of ghost projections for expanders.
Note however, that for our construction the acting group has to be neither residually finite nor with properties $(T)$ or ($\tau$) -- the existence 
of the ghost projection is solely a consequence of the spectral gap of the action. 
A spectral gap for a single representation is a weaker condition than property $(\tau)$.
There are many examples of actions with a spectral gap: see for instance \cite[Example 11.3]{nevo} 
for actions of non-amenable algebraic groups, and \cite{conze-guivarch} for actions on tori and nil-manifolds. Clearly, if $ G$ has property $(T)$ then every probability preserving action of $ G$ has a spectral gap.

\subsection{Warped cones as metric measure spaces}\label{sect:warped}

Let $(M,\dist ,m)$ be a compact metric space endowed with a probability measure, and let $G$ be a finitely generated group acting on $M$ by bi-Lipschitz homeomorphisms preserving the measure. Assume moreover that the action of $G$ is ergodic. Throughout this section we consider $G$ endowed with a finite, symmetric generating set $S$. 
We denote by $\vert g\vert $ the word length with respect to $S$ of $g\in G$.

We add a mild assumption on the measure, requiring in some sense its uniform distribution with respect to the metric. 

\begin{definition}\label{definition : upper uniform measure}
A measure $m$ on a metric space $(M,\dist )$ is called \emph{upper uniform} if 
$\lim_{R\to 0} \sup_{x\in M} m\left(B\left(x,R\right)\right) = 0.$
\end{definition}

Let $\operatorname{Cone}(M)= M \times (1,\infty)$ denote the truncated Euclidean cone over $M$, equipped with the measure $\nu$ which is a product measure of $m$ and the Le\-besgue measure. The restriction of the metric  $\dist_{\operatorname{Cone}(M)}$ on $\operatorname{Cone}(M)$  to $M\times \set{ t}$ is equal to $t \dist$.
The part corresponding to the interval $[0,1]$, both for the Euclidean cone, and for the warped cone defined below, is irrelevant for our purposes, since we 
are interested in large scale properties, and removing this part simplifies certain estimates.

The notion of a {\emph{warped cone}}, denoted by $\mathcal{O}=\mathcal{O}_{ G}(M)$, was first defined in \cite{roe-warped}. It is the space $M \times (1,\infty)$ endowed 
with the metric $\dist_\mathcal{O}$ that is the largest metric satisfying
$$\dist_{\mathcal{O}}(x,y)\le \dist_{\operatorname{Cone}(M)}(x,y)\ \ \ \ \text{ and }\ \ \ \ \dist_{\mathcal{O}}(x,sx)\le 1,$$
for every $x,y\in {\mathcal{O}}$ and $s\in S$. We endow $\mathcal{O}$ with the product measure $\nu$ of $m$ and the Lebesque measure.

For $t\ge 1$ we denote by ${\mathcal{O}}_t$ the part of the warped cone ${\mathcal{O}}$ that corresponds to $M\times (t,\infty)$.

In \cite{roe-warped} it is shown that the warped metric from $x$ to $y$ is the infimum over all sums
\begin{equation}\label{eq:warpedmetric}
\sum_{i=0}^{k-1} \dist(g_i x_i, x_{i+1})+\vert g_i\vert,
\end{equation}
taken over finite sequences $x=x_0,x_1,\dots,x_k=y$ in $M$ and $g_0,\dots,g_{k-1}$ in $G$. Moreover, if $\dist(x,y)_{\mathcal{O}}\le n$, 
where $n\in \N$, then we can choose $k\le n+1$.
Thus a warped cone is a metric space that interpolates between orbits of the action at $t=1$ and the group $G$ with the word metric at $t=\infty$.

If $\operatorname{Cone}(M)$ has bounded geometry (e.g. $M$ embeds into a finite-dimensional Euclidean space) (see \cite{nowak-yu,roe-lectures}) and
$ G$ acts on $\operatorname{Cone}(M)$ by bi-Lipschitz homeomorphisms then $\mathcal{O}_G(M)$ has bounded geometry \cite{roe-warped}. 

The following statement describes the relation 
of a ball in the warped metric to a ball in the Euclidean cone.
\begin{lemma}\label{lemma:RT}
Assume that $ G$ acts on $M$ by bi-Lipschitz maps. Then for each $R>0$ there exists $T>0$ such that every ball of radius $R$ in the warped cone ${\mathcal{O}}$, with center an arbitrary point  $x\in {\mathcal{O}}$, satisfies
$$B_{\mathcal{O}}(x,R)\subseteq \bigcup_{\vert g\vert \le R} B_{\operatorname{Cone}(M)}(g x, T).$$
\end{lemma}
\begin{proof} We use the definition of the warped metric as infimum of finite sums of the form described in \eqref{eq:warpedmetric}. Consider the case $g_0=e$. 
The other case is proved analogously and is omitted.
We have $x_1\in B(x_0, R)$. The next step in the chain is realized by $g_1x_1$, for some
group element $g_1\in B_ G(e,R_1)$, where in this case we set $R_1=R$. 
Then $x_2\in X$ is such that the inequality
$$\dist(x_2,g_1x_1)\le R_2=R-\dist(x_0,x_1)-\vert g_1\vert$$
is satisfied.
Observe now that 
\begin{align*}
\dist(g_1 x_0,x_2) & \le \dist(g_1 x_0,g_1 x_1)+\dist(g_1x_1,x_2)\\
&\le  L_{g_1}\dist(x_0,x_1)+ R_2\\
&\le L_{g_1} R_1+R_2,
\end{align*}
so that $x_2\in B(g_1x_0, L_{g_1}R_1+R_2)$, where $L_g$ denotes the Lipschitz constant of the transformation of $M$ associated to $g$. 
Following these estimates for $g_2$ and $x_3$ we observe that  $x_3\in B(g_2x_0, L_{g_2g_1} R_1+L_{g_2}R_2+R_3)$.
In general, for every $i=0,1,\dots, k$ there exists $T(g_i)$ such that $x_{i+1}\in B_{\operatorname{Cone}(M)}(g_ix, T(g_i))$,
where the radii $T(g_i)$ depend on $R$ and the Lipschitz constants of the transformations of $M$ associated to  $g_i$, but can be chosen independently of $x$. 
All possible choices for $g_i$ have to satisfy 
$\vert g_i\vert \le R$. Therefore setting $T=\sup \setd{T(g)}{\vert g \vert \le R}$ we obtain the claim.
\end{proof}
Note that if the action of $G$ on $M$ is isometric then we can take $T=R$ in Lemma \ref{lemma:RT}.
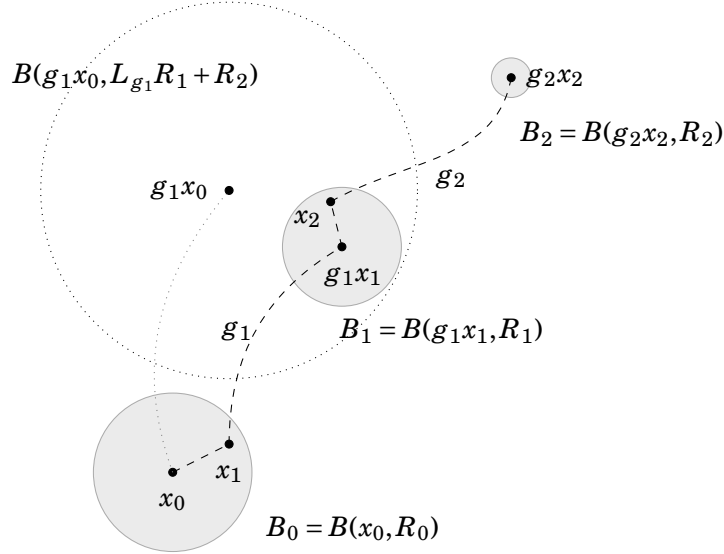
\begin{figure}[h]
\begin{center}
\begin{tikzpicture}[scale=0.75]
\draw[fill=lightgray,opacity=.3] (4,3) circle (40pt);
\draw[dashed] (4,3) to (5,3.5);
\filldraw (4,3) circle (2pt);
\filldraw (5,3.5) circle (2pt);
\node [below] at (4,2.8) {$x_0$};
\node [below] at (5,3.3) {$x_1$};

\draw[dotted] (5,8) circle (95pt);
\draw[gray,dotted] (4,3) to [out=110,in=230] (5,8);
\filldraw (5,8) circle (2pt);

\draw[fill=lightgray,opacity=.3] (7,7) circle (30pt);
\filldraw (7,7) circle (2pt);
\draw[dashed] (5,3.5) to [out=90,in=210] (7,7);

\draw[dashed] (7,7) to (6.8,7.8);
\filldraw (6.8,7.8) circle (2pt);

\draw [fill=lightgray,opacity=.3] (10,10) circle (10pt);
\filldraw (10,10) circle (2pt);

\draw[dashed] (6.8,7.8) to [out=30, in= 260] (10,10);

\node [below] at (7.2,6.8) {$g_1 x_1$};
\node [left] at (6.8,7.5) {$x_2$};
\node [right] at (10.1,10) {$g_2 x_2$};
\node [left] at (4.8,8) {$g_1 x_0$};

\node [right] at (5.5,2) {$B_0=B(x_0,R_0)$};
\node [right] at (6.8,5.5) {$B_1=B(g_1x_1,R_1)$};
\node [right] at (10,9) {$B_2=B(g_2x_2,R_2)$};
\node [right] at (1,10) {$B(g_1x_0, L_{g_1}R_1+R_2)$};

\node [right] at (4.7,5.5) {$g_1$};
\node [right] at (8.5,8.2) {$g_2$};

\end{tikzpicture}
\end{center}
\caption{A possible path realizing a distance $\le R$ in the warped metric. The length of this path 
is $\dist(x_0,x_1)+\vert g_1\vert +\dist(g_1x_1,x_2)+\vert g_2\vert\le R$.}
\end{figure}

\begin{lemma}\label{lemma: balls in the warped cone have small measure}
For every $R,\varepsilon >0$ there exists $t\in (1,\infty)$ such that 
$$
\nu \left( B_{\mathcal{O}}(x,R) \right)\le \varepsilon,
$$
 for every $x\in {\mathcal{O}}_t$.
\end{lemma}
\begin{proof}
Let $T>0$ and $x=(y,s)\in {\mathcal{O}}_t$. Every ball $B_{\operatorname{Cone(M)}}(x,T)$ is contained in a product $B(y,r) \times [s-T,s+T]$, where $r$ can be chosen with 
an upper bound depending only on $T$ and $s$.
Then 
$$\nu(B_{\operatorname{Cone(M)}}(x,T))\le 2T m(B(y,r)),$$
which tends to zero uniformly as $t\to \infty$, by  the upper uniformity of the measure $m$.
By the previous lemma we also have
$$\nu(B_{\mathcal{O}}(x,R)) \le \sum_{\vert g\vert \le R}\nu\left(B_{\operatorname{Cone}(M)}(g x,T)\right) \le 2T m(B(y,r )) \vert B_G(e,R)\vert, $$
which again tends to zero uniformly when $t\to \infty$, as $R$ and $T$ are fixed.
\end{proof}

\subsection{Finite propagation operators on a warped cone}
Consider a locally compact metric measure space $(X, \dist, m)$.
An \emph{$X$-module} is a (separable) Hilbert space $H$ equipped with a representation of $C_0(X)$.
An operator $T \in B(H)$ has {\emph{finite propagation}} if there exists $S>0$ such that for $\phi,\psi\in C_0(X)$ satisfying $\dist(\supp \phi, \supp \psi)>S$ we have
$\phi T\psi =0$.

The space $H=L_2({\mathcal{O}},\nu)$ is equipped with the standard, faithful representation of $C_0({\mathcal{O}})$ by multiplication operators
and thus naturally becomes an $\mathcal{O}$-module.

The action of $ G$ on ${\mathcal{O}}$ preserves the measure $\nu$  
and induces a unitary representation $\pi$ of $ G$ on $L_2({\mathcal{O}},\nu)\simeq L_2\left(M\times \left(1,\infty\right)\right)$ defined by
$\pi_g f(x,t)=f(g^{-1}x,t)$.

\begin{proposition}\label{proposition : pi has finite propagation}
For every $g\in G$ the operator $\pi_{g}$ has bounded propagation on $L_2(\mathcal{O},\nu)$. 

In particular, the Markov operator $A_\pi^{\mu}$ has bounded propagation 
for any choice of a probability 
measure $\mu$ supported on a finite generating set $S$ of $G$.
\end{proposition}
\begin{proof}
Assume that $\phi \pi_g \psi\neq 0$ for $\phi,\psi\in C_0({\mathcal{O}})$. This is possible only if 
$\supp \phi\cap g^{-1}(\supp \psi)\neq \emptyset$.
This however implies that $\dist_{\mathcal{O}}(\supp \phi,\supp \psi)\le \vert g\vert$.
\end{proof}

\subsection{Noncompact ghost projections}

The notion of ghost operator was introduced by G. Yu (unpublished).
\begin{definition}
Given a metric measure space $(X, \dist , \nu )$, an operator $T\in {\mathcal{B}}\left( L_2(X,\nu) \right)$ is said to be a \emph{ghost} if for every $R,\varepsilon>0$ there exists a bounded set $B\subseteq X$ such that 
for $\psi\in L_2(X,\nu)$ satisfying $\Vert \psi \Vert=1$ and $\supp \psi\subseteq B(x,R)$ for some $x\in X\setminus B$ we have $\Vert T\psi\Vert \le \varepsilon$.
\end{definition}

Ghost operators are operators that are \emph{locally invisible at infinity} \cite{chen-wang-1}. Such operators 
are intrinsically connected to large scale geometric features of the space \cite{chen-wang-1, chen-wang-2, roe-willett}.
Non-compact ghost projections are central objects in the context of the Baum-Connes conjecture, see below. We refer to 
\cite{higson,higson-lafforgue-skandalis,roe-lectures,willett-yu} for a detailed discussion.

Define $\mathfrak{G}:L_2(\mathcal{O},\nu)\to L_2(\mathcal{O},\nu)$ by setting 
$$\mathfrak{G}f(x,t)=\int_{M \times \set{t}} f(y,t)\, dm(y)$$
for every $(x,t)\in M \times (1,\infty )$.
The map $\mathfrak{G}$ is the orthogonal projection onto the subspace $V$ composed of functions that are constant on $ M \times \set{t}$ for every $t\in (1,\infty)$.
The subspace $V$ is a copy of $L_2(1,\infty)$ embedded in $L_2({\mathcal{O}},\nu)$.

\begin{theorem}\label{theorem: non-compact ghost projection}
Let $(M,\dist ,m)$ be a compact metric space endowed with a probability measure, and let $G$ be a finitely generated group acting on $M$ ergodically by bi-Lipschitz homeomorphisms preserving the measure.
If the action of $G$ on $M$ has a spectral gap then $\mathfrak{G}\in \mathcal{B}(L_2(\mathcal{O},\nu))$ is a non-compact ghost projection 
that is a norm limit of finite propagation operators.\end{theorem}

\begin{proof}
The projection $\mathfrak{G}$ is not compact since its range is infinite-dimensional.
The action has a spectral gap therefore, by 
Theorems \ref{theorem: main theorem in text spectral gaps and projections} and \ref{theorem: definition of the projection onto invariant vectors}, for every  
probability measure $\mu$ on $G$ admissible with respect to a finite generating set and of density function $\rho$, there exists $\lambda<1$ such that the following holds. For every $t>1$ we have 
$$\left\lVert \pi(\rho^k) f_{\vert M\times\set{t}} - \mathfrak{G}f(\cdot, t) \right \rVert_{L_2(M,m)}^2 \le \lambda^{2k} \left\lVert f_{\vert M\times\set{t}}\right\rVert_{L_2(M,m)}^2.$$
Using Fubini's theorem we conclude that
$$\left\lVert \pi(\rho^k) - \mathfrak{G}\right\rVert_{{\mathcal{B}}(L_2({\mathcal{O}},\nu))} \le \lambda^k,$$
where $\pi(\rho^k)=\left( A_\pi^{\mu}\right)^k$ has finite propagation by Proposition \ref{proposition : pi has finite propagation}.

It remains to show that $\mathfrak{G}$ is a ghost. Let $R,\varepsilon>0$.
For every $\delta >0$  there exists $t>0$ such that $\nu(B_{\mathcal{O}}(p,R))\le \delta$ for every $p\in {\mathcal{O}}_t$, by lemma \ref{lemma: balls in the warped cone have small measure}.
Consider $f\in L_2({\mathcal{O}},\nu)$ such that 
$\supp f \subseteq B(p,R)$
for some $p\in {\mathcal{O}}_t$. For every $s\in (t,\infty)$ the projection $\mathfrak{G}$ satisfies 
\begin{align*}
\mathfrak{G}f(x,s)^2 & =\left(\int_{M\times \set{s}} f(y,s)\, dm(y)\right)^2\\
&= m\left(\supp f\cap \left(M\times\set{s}\right)\right)^2\left(\dfrac{1}{m\left(\supp f\cap (M\times\set{s})\right)} \int_{\supp f\cap M\times\set{s}} f(y,s)\, dm(y)\right)^2\\
& \le m(\supp f\cap (M\times\set{s}))^2  \int_{\supp f\cap M\times\set{s}} f(y,s)^2\, dm(y)\\
&\le m(\supp f\cap (M\times\set{s}))^2.
\end{align*}
The above implies, by integration
and Fubini's theorem, that
$\Vert \mathfrak{G}f \Vert^2 \le \nu \left( \supp f \right )^2\le \nu(B_{\mathcal{O}}(p,R))^2\le \delta^2$.
\end{proof}

We point out that although Theorem \ref{theorem: non-compact ghost projection} is formulated for Hilbert spaces, the same proof gives a straightforward stronger
version. Namely, provided that the action on the Bochner space $L_2(M,m;E)$ has a spectral gap we obtain that the projection from 
$L_2(\mathcal{O},\nu;E)$ onto
$L_2([1,\infty);E)$ is a non-compact ghost projection which is a limit of finite propagation operators. At present, however, ghost projections on non-Hilbert spaces do not have
applications similar to the ones on Hilbert spaces. 
We also point out that very likely, the construction admits a generalization to a foliated versions of the warped cone \cite{roe-foliated}.

There are many group actions to which the above theorem applies, in particular actions on compact groups of finitely generated (free) subgroups. This latter type of constructions were motivated by the Ruziewicz problem (see \cite{bourgain-gamburd,bourgain-gamburd-jems} and references therein). Thus, Theorem \ref{theorem: non-compact ghost projection} applies for instance to certain warped cones $\mathcal{O}_{\mathbb{F}_n}(\operatorname{SU}(2))$, where the free group 
$\mathbb{F}_n$ is a subgroup of $\operatorname{SU}(2)$ generated by elements of a specific type \cite{bourgain-gamburd}.

\comment
In relation to the above, Guoliang Yu posed the following
\begin{question}
Does there exist a warped cone over an action with a spectral gap as above that does not coarsely contain a sequence of expanders?
\end{question}

The example of $\mathcal{O}_{\mathbb{F}_n}(\operatorname{SU}(2))$ mentioned previously seems to be a good candidate for such an example.
\endcomment

\comment
Note that warped cones satisfying the spectral gap property do not have Yu's property A  
(this property is in a sense a non-equivariant version of amenability; we refer to \cite{nowak-yu} for further details). 
This fact can be proved either by adapting the arguments in \cite{roe-warped} or with the following argument. 
Given a discrete space $X$ with property A, every ghost on $X$ is compact \cite{chen-wang-2} and an 
appropriate generalization of \cite[Lemma 4.3]{chen-wang-2} 
leads to a similar conclusion in our setting (in fact the existence of non-compact ghost operators characterizes the failure of property A, see \cite{roe-willett}). 
Therefore Theorem \ref{theorem: intro warped cones} allows to conclude that spectral gap representations yield warped cones without property A. 
\endcomment

\subsection{The coarse Baum-Connes conjecture}
Consider now an $X$-module $H$, which is ample; that is, no non-zero element of $C_0(X)$ acts on $H$ as a compact operator.
The Roe algebra $C^*(X)$ of a space $X$ is the closure of  locally compact finite propagation 
operators on $H$. 
An operator $T$ on an $X$-module $H$ is \emph{locally compact} if for every $f\in C_0(X)$  the operators $fT$ and $Tf$ are compact. 
Recall that the coarse Baum-Connes conjecture predicts that for a metric space $X$ of  bounded geometry, the coarse assembly map
\begin{equation}\label{equation: coarse assembly map}
\lim_{d\to\infty}\, K_*(P_d(X))\ \longrightarrow\  K_*(C^*(X)),
\end{equation}
from the coarse $K$-homology of $X$ to the $K$-theory of the Roe algebra, is an isomorphism. Above, $P_d(X)$ is the Rips complex at scale $d\ge 0$.
If true for a finitely generated group $G$ with a finite classifying space, the coarse Baum-Connes conjecture implies the Novikov conjecture on the homotopy invariance of higher signatures via a descent principle, see \cite[Theorem 8.4]{roe-cbms}.

Counterexamples to the coarse Baum-Connes conjecture were constructed in \cite{higson, higson-lafforgue-skandalis}.
It was proved there that the coarse assembly map is not surjective for $X$ a coarse disjoint union of expanders 
obtained as quotients of an exact group with property $(T)$. The reason is the existence of a non-compact ghost projection $\mathfrak{G}$
which is a limit of finite propagation operators.
The $K$-theory class represented by $\mathfrak{G}\otimes p$ in $K_*(C^*(X))$, where $p$ is a rank one projection, 
is not in the image of the coarse assembly map. At present such expanders provide
the only known bounded geometry counterexamples to the coarse Baum-Connes conjecture.

Let $G$, $M$ satisfy the assumptions of Theorem \ref{theorem: non-compact ghost projection}. 
\begin{conjecture}
The coarse assembly map (\ref{equation: coarse assembly map}) is not surjective for the warped cone $X=\mathcal{O}_GM$. 
\end{conjecture}


\end{document}